%% file: Main.tex
\documentclass[11pt,draftcls,onecolumn]{IEEEtran}

\renewcommand{\baselinestretch}{1.3}

\input{./LatexVariables/packages}

\input{./LatexVariables/commands}

\title{Synchronization and Transient Stability in Power Networks and
  Non-Uniform Kuramoto Oscillators\thanks{This work was supported in part
    by NSF grants IIS-0904501 and CNS-0834446. This document is a vastly
    revised and extended version of~\cite{FD-FB:09o}. The authors
    gratefully acknowledge Prof.\ Yoshihiko Susuki and Prof.\ Petar
    Kokotovi\'c for their insightful comments that improved the
    presentation of this article.  }}

\author{Florian D\"orfler,~\IEEEmembership{Member,~IEEE}, and Francesco
  Bullo,~\IEEEmembership{Fellow,~IEEE}
  \thanks{Florian D\"orfler and Francesco Bullo are with the Center for
    Control, Dynamical Systems and Computation, University of California at
    Santa Barbara, Santa Barbara, CA 93106, {\tt\small \{dorfler,
      bullo\}@engineering.ucsb.edu}} }

\begin{document}



\maketitle



\begin{abstract}
  Motivated by recent interest for multi-agent systems and smart power grid
  architectures, we discuss the synchronization problem for the
  network-reduced model of a power system with non-trivial transfer
  conductances.  Our key insight is to exploit the relationship between the
  power network model and a first-order model of coupled oscillators.
  Assuming overdamped generators (possibly due to local excitation
  controllers), a singular perturbation analysis shows the equivalence
  between the classic swing equations and a non-uniform Kuramoto
  model. Here, non-uniform Kuramoto oscillators are characterized by
  multiple time constants, non-homogeneous coupling, and non-uniform phase
  shifts.  Extending methods from transient stability, synchronization
  theory, and consensus protocols, we establish sufficient conditions for
  synchronization of non-uniform Kuramoto oscillators.  These conditions
  reduce to and improve upon previously-available tests for the standard
  Kuramoto model. Combining our singular perturbation and Kuramoto
  analyses, we derive concise and purely algebraic conditions that relate
  synchronization and transient stability of a power network to the
  underlying system parameters and initial conditions.
\end{abstract}



\section{Introduction}\label{Section: Introduction}
\input{./Sections/Introduction}

\section{Models and Problem Setup in Synchronization and Transient
  Stability Analysis}
\label{Section: Classical Transient Stability Analysis}
\input{./Sections/Power_Network_Mathematical_Model}


\section{The Non-Uniform Kuramoto Model and Main Synchronization Result}\label{Section: The Non-Uniform Kuramoto Model and Main Synchronization Result}

\input{./Sections/MainResult}


\clearpage
\section{Singular Perturbation Analysis of Synchronization}
\label{Section: Singular Perturbation Analysis of Synchronization}
\input{./Sections/Singular_Perturbation_Analysis}


\section{Synchronization of Non-uniform Kuramoto Oscillators}
\label{Section: Synchronization of Non-uniform Kuramoto System}
\input{./Sections/FrequencyEntrainment}

\input{./Sections/PhaseClustering_1}

\input{./Sections/PhaseClustering_2}
\input{./Sections/PhaseSynchronization}


\section{Simulation Results}\label{Section: Simulation Results}
\input{./Sections/Simulations}


\section{Conclusions}\label{Section: Conclusions}

\input{./Sections/Conclusions}





\renewcommand{\baselinestretch}{1.261}

\bibliographystyle{./LatexVariables/IEEEtran}
\bibliography{alias,Main,FB}



\section{Appendix: Alternative Synchronization Conditions}\label{Appendix}

\input{./Sections/Appendix}

\end{document}

%% file: LatexVariables/packages.tex
\usepackage{graphicx}          		
\usepackage{graphics}         	 	
\usepackage{epsfig}            		
\usepackage{amsmath}           		
\usepackage{amssymb}           	        
\usepackage[usenames,dvipsnames]{color}   
\usepackage{hyperref}          		
\usepackage{footmisc}      		
\usepackage{subfigure}        		
\usepackage{algorithm2e}			

\definecolor{MyDarkBlue}{rgb}{0,0.08,0.45}
\hypersetup{
    bookmarks=true,             		
    unicode=false,              		
    pdftoolbar=true,            		
    pdfmenubar=true,            		
    pdffitwindow=true,          		
    pdftitle={My title},        			
    pdfauthor={Author},         		
    pdfsubject={Subject},       		
    pdfnewwindow=true,          		
    pdfkeywords={keywords},     	
    colorlinks=true,            		
    linkcolor=MyDarkBlue,       		
    citecolor=RoyalPurple,     		
    filecolor=magenta,          		
    urlcolor=green              		
}

%% file: LatexVariables/commands.tex
\newcommand*\fvec[1]{\ensuremath{\mathbf{#1}}}                                	
\newcommand*\dx[3]{\frac{\partial^{#3} {#1}}{\partial {#2}^{#3}}}			
\newcommand*\dt[0]{\frac{d}{d\,t}\,}					                      	
\newcommand{\abs}[1]{\left\lvert{#1}\right\rvert}                     			
\newcommand{\norm}[1]{\left\lVert{#1}\right\rVert}                     		
\newcommand*\mc[0]{\mathcal}                                                  		
\newcommand*\mbb[0]{\mathbb}                                                  		
\DeclareMathOperator*{\argmin}{argmin}                                        		
\DeclareMathOperator*{\sinc}{sinc} 	                                       		
\DeclareMathOperator*{\sinbf}{\mathrm{\bf sin}} 	                              
\DeclareMathOperator*{\diag}{\mathrm{diag}}						

\usepackage{theorem}
\newtheorem{theorem}{Theorem}[section]

\newtheorem{lemma}[theorem]{Lemma}
{\theorembodyfont{\rmfamily} 
\newtheorem{remark}[theorem]{Remark}

\newcommand{\until}[1]{\{1,\dots, #1\}}
\newcommand{\subscr}[2]{#1_{\textup{#2}}}
\newcommand{\supscr}[2]{#1^{\textup{#2}}}
\newcommand{\setdef}[2]{\{#1 \; | \; #2\}}

\newcommand{\map}[3]{#1: #2 \rightarrow #3}

\newcommand\oprocendsymbol{\hbox{$\square$}}
\newcommand\oprocend{\relax\ifmmode\else\unskip\hfill\fi\oprocendsymbol}

\newcommand{\real}{\mathbb{R}}

\newcommand{\groundedphases}{\operatorname{grnd}}

%% file: Sections/Introduction.tex
The vast North American interconnected power grid is often referred to as
the largest and most complex machine engineered by humankind. The 
various instabilities arising in such a large-scale power grid can be classified 
by their physical nature, the size of the uncertainty or disturbance causing 
the instability, or depending on the
devices, processes, and the time necessary to determine the instability.
All of these instabilities can lead and have led to blackouts of power
grids \cite{PP-PSK-CMT:06}, and their detection and rejection will be one of
the major challenges faced by the future ``smart power grid.''

The envisioned future power generation will rely increasingly on renewable
energy sources such as wind and solar power. Since these renewable power
sources are highly stochastic, there will be an increasing number of
transient disturbances acting on a power grid that is expected to be even
more complex and decentralized than the current one. Thus, an important
form of power network stability is the so-called {\it transient stability}
\cite{VV:03}, which is the ability of a power system to remain in
synchronism when subjected to large transient disturbances. These
disturbances may include faults on transmission elements or loss of load,
loss of generation, or loss of system components. For example, a recent
major blackout in Italy in 2003 was caused by tripping of a tie-line and
resulted in a cascade of events leading to the loss of synchronism of the
Italian power grid with the rest of Europe \cite{PP-PSK-CMT:06}. The
mechanism by which interconnected synchronous machines maintain synchronism
is a balance of their mechanical power inputs and their electrical power
outputs depending on the relative rotor angles among machines.  In a
classic setting the transient stability problem is posed as a special case
of the more general {\it synchronization\,\,problem}, which is defined over
a possibly longer time horizon, for rotor angles possibly drifting away
from their nominal values, and for generators subject to local excitation
controllers aiming to restore synchronism. In order to analyze the
stability of a synchronous operating point of a power grid and to estimate
its region of attraction, various sophisticated algorithms have been
developed \cite{HDC-FFW-PPV:88,TA-RP-SV:79,JB-CIB:82,
  HDC-CCC:95, NGB-LFCA:03,FS-LFCA-JBAL-NGB:05}.
Reviews and survey articles on transient stability analysis can be found in
\cite{PV-FFW-RLC:85,MAP:89,HDC-CCC-GC:95,LFCA-FS-NGB:01}.
Unfortunately, the existing methods can cope only with 
simplified models and do not provide simple formulas to check if a power
system synchronizes for a given system state and parameters. In fact, an
open problem, recognized by \cite{DJH-GC:06} and not resolved by classical
analysis methods, is the quest for explicit and concise conditions for
synchronization as a function of the topological, algebraic, and spectral
graph properties of the network.


The recent years have witnessed a burgeoning interest of the control
community in cooperative control of autonomous agent systems. Recent
surveys and monographs
include~\cite{ROS-JAF-RMM:07,WR-RWB-EMA:07,FB-JC-SM:09}. One of the basic
tasks in a multi-agent system is a consensus of the agents' states to a
common value. This consensus problem has been subject to fundamental
research \cite{LM:05,ZL-BF-MM:07} as well as to applications in robotic
coordination, distributed sensing and computation, and various other fields
including synchronization. In most articles treating consensus problems the
agents obey single integrator dynamics, but the synchronization of
interconnected power systems has often been envisioned as possible future
application \cite{DJH-JZ:08}. However, we are aware of only one
article~\cite{MA:07} that indeed applies consensus methods to a power
network model.

Another set of literature relevant to our investigation is the
synchronization of coupled oscillators \cite{AA-ADG-JK-YM-CZ:08}, in
particular in the classic model introduced by Kuramoto \cite{YK:03}. The
synchronization of coupled Kuramoto oscillators has been widely studied by
the physics \cite{FDS-DA:07,JLvH-WFW:93,KW-PC-SHS:98} and the dynamical
systems communities \cite{REM-SHS:05,EC-PM:08,DA-JAR:04}.  This vast
literature with numerous theoretical results and rich applications to
various scientific areas is elegantly reviewed in
\cite{SHS:00,JAA-LLB-CJPV-FR-RS:05}.  Recent works in the control community
\cite{LM:05,ZL-BF-MM:07,NC-MWS:08,AJ-NM-MB:04} investigate the close
relationship between Kuramoto oscillators and consensus networks.

The three areas of power network synchronization, Kuramoto oscillators, and
consensus protocols are apparently closely related.  Indeed, the similarity
between the Kuramoto model and the power network models used in transient
stability analysis is striking.  Even though power networks have often been
referred to as systems of coupled oscillators, the similarity to a
second-order Kuramoto-type model has been mentioned only very recently in
the power networks community in
\cite{GF-AHN-NFP:08,VF-SR-EC-EM-VR:09,DS-UR-BS-MP:01}, where only
qualitative simulation studies for simplified models are carried out. In
the coupled-oscillators literature, second-order Kuramoto models similar to
power network models have been analyzed in simulations and in the continuum
limit; see \cite{JAA-LLB-CJPV-FR-RS:05} and references therein. However, we
are aware of only two articles referring to power networks as possible
application \cite{AA-ADG-JK-YM-CZ:08,HAT-AJL-SO:97}.  In short, the evident
relationship between power network synchronization, Kuramoto oscillators,
and consensus protocols has been recognized, but the gap between the first
and the second two topics has not been bridged yet in a thorough analysis.
   
There are three main contributions in the present paper.  As a first
contribution, we present a coupled-oscillator approach to the problem of
synchronization and transient stability in power networks. Via a singular
perturbation analysis, we show that the transient stability analysis for
the classic swing equations with overdamped generators reduces, on a long
time-scale, to the problem of synchronizing non-uniform Kuramoto
oscillators with multiple time constants, non-homogeneous coupling, and
non-uniform phase-shifts. This reduction to a non-uniform Kuramoto model is
arguably the missing link connecting transient stability analysis to
networked control, a link that was hinted at in \cite{DJH-GC:06, DJH-JZ:08,
  GF-AHN-NFP:08, VF-SR-EC-EM-VR:09, DS-UR-BS-MP:01, AA-ADG-JK-YM-CZ:08,
  MA:07}.

Second, we give novel, simple, and purely algebraic conditions that are
sufficient for synchronization and transient stability of a power
network. To the best of our knowledge these conditions are the first ones
to relate synchronization and performance of a power network directly to
the underlying network parameters and initial state.
Our conditions are based on different and possibly less restrictive
assumptions than those obtained by classic analysis methods
\cite{HDC-FFW-PPV:88,TA-RP-SV:79,JB-CIB:82,HDC-CCC:95,NGB-LFCA:03,FS-LFCA-JBAL-NGB:05,PV-FFW-RLC:85}.  We consider a
network-reduced model of a power network, and do not make any of the
following common or classic assumptions: we do not require the swing
equations to be formulated in relative coordinates accompanied by a uniform
damping assumption, we do not require the existence of an infinite bus, and
we do not require the transfer conductances to be ``sufficiently small'' or
even negligible.  On the other hand, our results are based on the
assumption that each generator is strongly overdamped, possibly due to
internal excitation control. This assumption allows us to perform a
singular perturbation analysis and study a dimension-reduced system.  Due
to topological equivalence, our synchronization conditions hold locally
even if generators are not overdamped, and in the application to real power
networks the approximation via the dimension-reduced system is
theoretically well-studied and also applied in the power industry
\cite{H-DC:10}.  Our synchronization conditions are based on an analytic
approach whereas classic analysis methods
\cite{HDC-FFW-PPV:88,TA-RP-SV:79,HDC-CCC:95,NGB-LFCA:03,FS-LFCA-JBAL-NGB:05,PV-FFW-RLC:85} rely on numerical procedures
to approximate the region of attraction of an equilibrium by level sets of
energy functions and stable manifolds.  Compared to classic analysis
methods, our analysis does not aim at providing best estimates of the
region of attraction or the critical clearing time. Rather, we approach the
open problem~\cite{DJH-GC:06} of relating synchronization to the underlying
network structure. For this problem, we derive sufficient and purely
algebraic conditions that can be interpreted as ``the network connectivity
has to dominate the network's non-uniformity, the network's losses, and the
lack of phase cohesiveness.''


Third and final, we perform a synchronization analysis of non-uniform
Kuramoto oscillators, as an interesting mathematical problem in its own
right. Our analysis combines and extends methods from consensus protocols
and synchronization theory. As an outcome, purely algebraic conditions on
the network parameters and the system state establish the 
phase cohesiveness, frequency synchronization, and phase
synchronization of the non-uniform Kuramoto oscillators. We emphasize that
our results do not hold only for non-uniform network parameters but also in
the case when the underlying coupling topology is not a complete
graph. When our results are specialized to classic (uniform) Kuramoto
oscillators, they reduce to and even improve upon various well-known
conditions in the literature on the Kuramoto model
\cite{JLvH-WFW:93,FDS-DA:07,ZL-BF-MM:07,NC-MWS:08,AJ-NM-MB:04,GSS-UM-FA:09}. In
the end, these conditions guaranteeing synchronization of non-uniform
Kuramoto oscillators also suffice for the transient stability of the power
network.

\subsubsection*{Paper Organization}
This article is organized as follows.  The remainder of this section
introduces some notation, recalls preliminaries on algebraic graph theory
and differential geometry, and reviews the consensus protocol and the Kuramoto model of
coupled oscillators.  Section \ref{Section: Classical Transient Stability
  Analysis} reviews the problem of transient stability analysis. Section
\ref{Section: The Non-Uniform Kuramoto Model and Main Synchronization
  Result} introduces the non-uniform Kuramoto model and presents the main
result of this article. Section \ref{Section: Singular Perturbation
  Analysis of Synchronization} translates the power network model to the
non-uniform Kuramoto model whose synchronization analysis is presented in
Section \ref{Section: Synchronization of Non-uniform Kuramoto
  System}. Section \ref{Section: Simulation Results} provides simulation
studies to illustrate the analytical results. Finally, some conclusions are
drawn in Section \ref{Section: Conclusions}. The appendix in Section \ref{Appendix} contains different synchronization conditions and estimates for the phase cohesiveness that can be derived alternatively to the ones presented in Section \ref{Section: Synchronization of Non-uniform Kuramoto System}.

\subsubsection*{Vector and matrix notation}

Given an $n$-tuple $(x_{1},\dots,x_{n})$, $\diag(x_{i}) \in \mbb R^{n
  \times n}$ is the associated diagonal matrix, $x \in \mbb R^{n}$ is the
associated column vector, $\subscr{x}{max}$ and $\subscr{x}{min}$ are the
maximum and minimum elements, and $\norm{x}_{2}$ and $\norm{x}_{\infty}$
are the $2$- and $\infty$-norm. Let $\fvec 1_{n}$ and $\fvec 0_{n}$ be the vectors
of $1$'s and $0$'s of dimension $n$. Given two non-zero vectors $x
\in \mbb R^{n}$ and $y \in \mbb R^{n}$, the angle $\angle(x,y) \in
[0,\pi/2]$ between them satisfies $\cos(\angle(x,y)) = x^{T}y /(||x|| \,
||y||)$. Given an array $\{A_{ij}\}$ with $i ,j \in \until n$, $A \in \mbb
R^{n \times n}$ is the associated matrix with $\subscr{A}{max} = \max_{i,j}
\{A_{ij} \}$ and $\subscr{A}{min} = \min_{i,j} \{A_{ij} \}$. Given a total
order relation among the indices $(i,j)$, let $\diag(A_{ij})$ denote the
corresponding diagonal matrix.

\subsubsection*{Graph theory}

A {\it weighted directed graph} is a triple $\mc G = (\mc V,\mc E,A)$,
where $\mc V = \until{n}$ is the set of nodes, $\mc E \subset \mc V \times
\mc V$ is the set of directed edges, and $A \in \mbb R^{n \times n}$ is the
{\it adjacency matrix}. The entries of $A$ satisfy $a_{ij} > 0$ for each
directed edge $(i,j) \in \mc E$ and are zero otherwise. Any nonnegative
matrix $A$ induces a weighted directed graph $\mc G$. The {\it Laplacian}
of $\mc G$ is the $n \times n$ matrix $L(a_{ij}) :=
\mathrm{diag}(\sum_{j=1}^na_{ij}) -A$.\,\,In the following, we assume that
$A=A^{T}$, that is, $\mc G$ is undirected. In this case, the graph $\mc G$
is fully described by the elements $a_{ij}$ with $i\geq{j}$. If a number $k
\in \until{|\mc E|}$ and a weight $w_{k} = a_{ij}$ is assigned to any of these
edges $(i,j)$ with $i>j$, then the \textit{incidence matrix} $H \in \mathbb
R^{|\mc E| \times n}$ is defined component-wise as $H_{kl} = 1$ if node
$l$ is the sink node of edge ${k}$ and as $H_{kl} = -1$ if node $l$ is the
source node of edge ${k}$; all other elements are zero. The Laplacian
equals then the symmetric matrix $L(a_{ij}) = H^{T} \diag(w_{k}) H$. If
$\mc G$ is connected, then $\mathrm{ker}(H) = \mathrm{ker}(L(a_{ij})) =
\mathrm{span}(\fvec 1_{n})$, all $n-1$ remaining non-zero eigenvalues of $L(\mc
G)$ are strictly positive, and the second-smallest eigenvalue
$\lambda_{2}(L(a_{ij}))$ is called the {\it algebraic connectivity} of $\mc
G$ and, for a complete and uniformly weighted graph ($a_{ij} \equiv 1$ for
all $i \neq j$), it satisfies $\lambda_{2}(L(a_{ij}))=n$.

\subsubsection*{Geometry on the $n$-torus}
\label{subsec:geom-torus}

The {\it torus} is the set $\mbb T^{1} = {]\!-\!\pi,+\pi]}$, where $-\pi$
and $+\pi$ are associated with each other, an {\it angle} is an element
$\theta \in \mbb T^{1}$, and an {\it arc} is a connected subset of $\mbb
T^{1}$. The product set $\mbb T^{n}$ is the $n$-dimensional torus.  With
slight abuse of notation, let $|\theta_{1}-\theta_{2}|$ denote the {\it
  geodesic distance} between two angles $\theta_{1} \in \mbb T^{1}$ and
$\theta_{2} \in \mbb T^{1}$.  For $\gamma\in{[0,\pi]}$, let $\Delta(\gamma)
\subset \mbb T^{n}$ be the set of angle arrays
$(\theta_1,\dots,\theta_n)\in\mbb T^{n}$ such that there exists an arc of
length $\gamma$ containing all $\theta_1,\dots,\theta_n$ in its
interior. Thus, an array of angles $\theta \in \Delta(\gamma)$ satisfies
$\max\nolimits_{i,j \in \until n} |\theta_{i} - \theta_{j}| < \gamma$.   
For $\gamma\in{[0,\pi]}$, we also define $\bar\Delta(\gamma)$ to
be the union of the set $\{\theta \in \mbb T^{n} \;|\;
\theta_{i} = \theta_{j} ,\, i,j \in \until n\}$ and the closure of the open
set $\Delta(\gamma)$.

For a rigorous definition of the {\it difference} between angles (i.e.,
points on the torus), we restrict our attention to angles contained in an
open half-circle: for angles $\theta_{1}$, $\theta_{2}$ with $|\theta_{1} -
\theta_{2}|<\pi$, the difference $\theta_{1}-\theta_{2}$ is the number in
${]\!-\!\pi,\pi[}$ with magnitude equal to the geodesic distance
$|\theta_{1} - \theta_{2}|$ and with positive sign iff the
counter-clockwise path length connecting $\theta_{1}$ and $\theta_{2}$ on
$\mbb T^{1}$ is smaller than the clockwise path length.  Finally, we define
the multivariable sine $\map{\sinbf}{\mbb T^{n}}{[0,1]^n}$ by $\sinbf(x) =
(\sin(x_{1}),\dots,\sin(x_{n}))$ and the sinc function
$\map{\sinc}{\real}{\real}$ by $\sinc(x)=\sin(x)/x$.

\subsubsection*{Review of the Consensus Protocol and the Kuramoto
  Model}\label{Subsection: The Consensus Protocol and the Kuramoto Model}

In a system of $n$ {\it autonomous agents}, each characterized by a state
variable $x_{i}\in \mbb R$, one of the most basic tasks is to achieve a
consensus on a common state value, that is, all agent states $x_{i}(t)$ converge to a common value $x_{\infty} \in \mbb R$ as $t
\to \infty$. Given a graph $\mc G$ with
adjacency matrix $A$ describing the interaction between agents,
a simple, linear, and continuous time algorithm to achieve consensus on the
agents' state is the {\it consensus protocol}
\begin{equation}
	\dot x_{i}
	=
	- \sum\nolimits_{j=1}^{n} a_{ij} (x_{i}-x_{j})
	, \quad
	i \in \until n
	\label{eq: Consensus protocol}
	\,.
\end{equation}
In vector notation the consensus protocol \eqref{eq: Consensus protocol}
takes the form $\dot{x} = - L(a_{ij}) x$, which directly reveals the
dependence of the consensus protocol to the underlying graph $\mc G$.

A well-known and widely used model for the synchronization among coupled
oscillators is the {\it Kuramoto model}, which considers $n$ coupled
oscillators with state $\theta_{i}\in\mbb T^1$ with the dynamics
\begin{equation}
  \dot \theta_{i}
  =
  \omega_{i} - \frac{K}{n} \sum\nolimits_{j=1}^{n} \sin (\theta_{i} - \theta_{j})
  , \quad
  i \in \until{n}
  \label{eq: Kuramoto system}
  \,,
\end{equation}
where $K$ is the coupling strength and $\omega_{i}$ is the natural
frequency of oscillator $i$. Unlike for the consensus protocol \eqref{eq:
  Consensus protocol}, different levels of consensus or {\it synchronization} can be distinguished for the Kuramoto model \eqref{eq:
  Kuramoto system}: The case when all angles $\theta_{i}(t)$ converge to a common angle $\theta_{\infty} \in \mbb T^{1}$ as $t \to \infty$ is referred to as {\it phase synchronization} and can only occur if all natural frequencies are identical. If the natural
frequencies are non-identical, then each phase difference $\theta_{i}(t) - \theta_{j}(t)$ can converge to a constant value, but this value is
not necessarily zero. 
A solution $\map{\theta}{\real_{\geq0}}{\mbb T^n}$ to the Kuramoto
model~\eqref{eq: Kuramoto system} is \emph{phase cohesive} if there exists
a length $\gamma \in [0,\pi[$ such that $\theta(t)\in\bar\Delta(\gamma)$ for all
$t\geq0$, i.e., at each time $t$ there exists an arc of length $\gamma$
containing all angles $\theta_i(t)$.  A solution
$\map{\theta}{\real_{\geq0}}{\mbb T^n}$ achieves \emph{exponential frequency synchronization} if all frequencies $\dot \theta_{i}(t)$ converge
exponentially fast to a common frequency $\dot \theta_{\infty} \in \mbb R$ as $t \to \infty$.  Finally, a solution
$\map{\theta}{\real_{\geq0}}{\mbb T^n}$ achieves {\it exponential
  synchronization} if it is phase cohesive and it achieves exponential
frequency synchronization.
In this case, all phases become constant in a rotating\,\,coordinate frame with frequency $\dot \theta_{\infty}$, and
hence the terminology {\it phase locking} is sometimes also used in the literature.

%% file: Sections/Power_Network_Mathematical_Model.tex
\subsection{The Mathematical Model of a Power Network}\label{Subsection: The Mathematical Model of a Power Network}

In a power network with $n$ generators we associate to each generator its
internal voltage $E_{i}>0$, its active power output $P_{\textup{e},i}$, its
mechanical power input $P_{\textup{m},i}>0$, its inertia $M_{i}>0$, its
damping constant $D_i>0$, and its rotor angle $\theta_{i}$ measured with
respect to a rotating frame with frequency $f_{0}$. All parameters are
given in per unit system, except for $M_{i}$ and $D_i$ which are given in
seconds, and $f_{0}$ is typically given as $50$ Hz or $60$ Hz. The rotor
dynamics of generator $i$ are then given by the classic constant-voltage
behind reactance model of interconnected {\it swing equations}
\cite{MAP:89,PMA-AAF:77,PK:94}
\begin{align*}
	M_{i} \ddot \theta_{i}
	&=
	P_{\textup{m},i} - E_{i}^{2}G_{ii} - D_{i} \dot{\theta}_{i} - P_{\textup{e},i} 
	,\quad
	i \in \until{n}
	\,.
\end{align*}
Under the common assumption that the loads are modeled as passive
admittances, all passive nodes of a power network can be eliminated
(c.f. {\it Kron reduction} \cite{FD-FB:11d}) resulting in the {\it reduced (transfer)
  admittance matrix} $Y=Y^{T} \in \mbb C^{n \times n}$, where $Y_{ii}$ is
the self-admittance of generator $i$ and $\Re(Y_{ij})\geq0$ and
$\Im(Y_{ij})>0$, $i \neq j$, are the transfer conductance and (inductive)
transfer susceptance between generator $i$ and $j$ in per unit values.
With the {\it power-angle relationship}, the active output power
$P_{\textup{e},i}$ is then
\begin{equation*}
  P_{\textup{e},i} 
  =
  \sum\nolimits_{j=1}^{n} E_{i} E_{j} \bigl( \Re(Y_{ij}) \cos(\theta_{i} - \theta_{j}) + \Im(Y_{ij}) \sin(\theta_{i} - \theta_{j}) \bigr).
\end{equation*}
Given the transfer admittance $Y_{ij}$ between generator $i$ and $j$,
define the magnitude $\abs{Y_{ij}} > 0$ and the \emph{phase shift}
$\varphi_{ij} = \arctan(\Re(Y_{ij})/\Im(Y_{ij})) \in {[0,\pi/2[}$ depicting
the energy loss due to the transfer conductance $\Re(Y_{ij})$.  Recall that
a lossless network is characterized by zero phase shifts.  Furthermore, we
define the {\it natural frequency} $\omega_{i} := P_{\textup{m},i} -
E_{i}^{2}\Re(Y_{ii})$ (effective power input to generator $i$) and the {\it
  coupling weights} $P_{ij} := E_{i}E_{j}|Y_{ij}|$ (maximum power
transferred between generators $i$ and $j$) with $P_{ii} :=0$ for $i \in
\until{n}$. The {\it network--reduced} power system model can then be
formulated compactly as
\begin{equation}
	M_{i} \ddot \theta_{i}
	=
	- D_{i} \dot{\theta}_{i} + \omega_{i}
	-
	\sum\nolimits_{j=1}^{n} P_{ij} \sin(\theta_{i} - \theta_{j} + \varphi_{ij})
	\label{eq: Classical model}
	\,.
\end{equation}
Typically, a dynamical model for the internal
voltage of generator $i$ is given as $ \dot{E}_{i} =
\dot{E}_{i}(E_{i},u_{i},\theta_{i}-\theta_{j}) $, where $u_{i}$ is the
field excitation and can be used as a control input~\cite{RO-MG-AA-YS-TS:05}.
Higher order electrical and flux dynamics can be reduced~\cite{PWS-MAP:98}
into an augmented damping constant $D_{i}$ in equation~\eqref{eq: Classical
  model}. The generator's internal excitation control essentially increases
the {\it damping torque} towards the net frequency and can also be reduced
into the damping constant $D_{i}$ \cite{PMA-AAF:77,PWS-MAP:98}. It is
commonly agreed that the classical model \eqref{eq: Classical model}
captures the power system dynamics sufficiently well during the first
swing. Thus, we omit higher order dynamics and control effects and assume
they are incorporated into the model \eqref{eq: Classical model}.  We
remark\,that all our results are also valid if $E_{i} = E_{i}(t)$ is a
smooth, bounded, and strictly positive time-varying parameter.



A {\it frequency equilibrium} of \eqref{eq: Classical model} is
characterized by $\dot{\theta} = \fvec 0$ and by the (reduced) {\it real power flow
  equations}
\begin{equation}
	Q_{i}(\theta)
	:= 
	\omega_{i}
	-
	\sum\nolimits_{j=1}^{n} P_{ij} \sin(\theta_{i} - \theta_{j} + \varphi_{ij})
	\equiv
	0
	,\quad
	i \in \until{n}
	\,.
	\label{eq: Classical model - power flow}
\end{equation}
depicting the power balance. More general, the generators are said to be in
a {\it synchronous equilibrium} if all angular distances
$|\theta_{i}-\theta_{j}|$ are constant and bounded (phase cohesive) and all frequencies are identical $\dot{\theta}_{i} = \dot{\theta}_{j}$. {\it Exponential
  synchronization} is then understood as defined before for the Kuramoto
model \eqref{eq: Kuramoto system}.  

In order to analyze the synchronization problem,
system \eqref{eq: Classical model} is usually formulated in relative
coordinates \cite{LFCA-NGB:99}.
To render the resulting dynamics self-contained, uniform damping is sometimes
assumed, i.e., $D_{i}/M_{i}$ is constant.  Some other times, the existence
of an infinite bus (a stationary generator without dynamics) as reference
is postulated \cite{HDC-FFW-PPV:88,PV-FFW-RLC:85}. We remark
that both of these assumptions are not physically justified but are
mathematical simplifications to reduce the synchronization problem to a
stability analysis.

\subsection{Review of Classic Transient Stability Analysis}
\label{Subsection: Review of Classical Transient Stability Analysis}

Classically, transient stability analysis deals with a special case of the
synchronization problem, namely the stability of a post-fault frequency
equilibrium, that is, a new equilibrium of \eqref{eq: Classical model}
arising after a change in the network parameters or topology.
To answer this question various sophisticated analytic and numeric methods
have been developed
\cite{PV-FFW-RLC:85,MAP:89,HDC-CCC-GC:95,LFCA-FS-NGB:01}, which typically
employ the Hamiltonian structure of system \eqref{eq: Classical
  model}. Since in general a Hamiltonian function for model \eqref{eq:
  Classical model} with non-trivial network conductance $\Re(Y_{ij})>0$ (or
equivalently $\varphi_{ij} > 0$) does not exist \cite{HDC:89}, early
transient stability approaches neglect the phase shifts $\varphi_{ij}$
\cite{HDC-FFW-PPV:88,JB-CIB:82,PV-FFW-RLC:85}. In this case, the power network model
\eqref{eq: Classical model} takes form
\begin{equation}
	M \, \ddot \theta 
	= 
	- D \dot \theta - \nabla U(\theta)^{T}
	\label{eq: Classical model - gradient system}
	\,,
\end{equation} 
where $\nabla$ is the gradient and $\map{U}{]-\pi,\pi]^n}{\real}$ is the potential energy given up to an additive constant\,\,by
\begin{equation}
  U(\theta)
  =
  - \sum\nolimits_{i=1}^{n} \Bigl( \omega_{i} \theta_{i}
  +  \sum\nolimits_{j=1}^{n} P_{ij} \bigl(1 - \cos(\theta_{i} -\theta_{j}) \bigr) \Bigr)
  \label{eq: potential energy}
  \,.
\end{equation}
When system \eqref{eq: Classical model - gradient system} is formulated in
relative or reference coordinates (that feature equilibria), the {\it
  energy function} $ (\theta,\dot{\theta}) \mapsto (1/2) \,
\dot{\theta}^{T} M \dot{\theta} + U (\theta) $ serves
(locally) as a Lyapunov function. In combination with the invariance
principle, we clearly have that the dynamics \eqref{eq: Classical model -
  gradient system} converge to $\dot{\theta} = \fvec 0$ and the largest
invariant zero level set of $\nabla U (\theta)$. In order to estimate the
region of attraction of a stable equilibrium, algorithms such as {\it PEBS}
\cite{HDC-FFW-PPV:88} or {\it BCU} \cite{HDC-CCC:95} consider the
associated dimension-reduced gradient flow
\begin{equation}
	\dot{\theta}
	=
	- \nabla U(\theta)^{T}
	\label{eq: Classical model - gradient system - reduced model}
	\,.
\end{equation}
Then $(\theta^{*},\fvec 0)$ is a hyperbolic type-$k$ equilibrium of
\eqref{eq: Classical model - gradient system}, i.e., the Jacobian has $k$
stable eigenvalues, if and only if $\theta^{*}$ is a hyperbolic type$-k$ equilibrium of \eqref{eq:
  Classical model - gradient system - reduced model}, and if a generic
transversality condition holds, then the regions of attractions of both
equilibria are bounded by the stable manifolds of the same unstable
equilibria \cite[Theorems 6.2-6.3]{HDC-FFW-PPV:88}. This topological
equivalence between \eqref{eq: Classical model - gradient system} and
\eqref{eq: Classical model - gradient system - reduced model} can also be
extended to ``sufficiently small'' transfer conductances \cite[Theorem
5.7]{HDC-CCC:95}. For further interesting relationships among the systems
\eqref{eq: Classical model - gradient system} and \eqref{eq: Classical
  model - gradient system - reduced model}, we refer to
\cite{HDC-FFW-PPV:88,HDC-CCC:95,HDC-CCC-GC:95,PV-FFW-RLC:85}.
Other approaches to lossy power networks with non-zero transfer
conductances compute numerical energy functions \cite{TA-RP-SV:79} or
employ an extended invariance principle \cite{FS-LFCA-JBAL-NGB:05}. Based
on these results computational methods were developed to approximate the
stability boundaries of \eqref{eq: Classical model - gradient system} and
\eqref{eq: Classical model - gradient system - reduced model} by level sets
of energy functions or stable manifolds of unstable equilibria.

To summarize the shortcomings of the classical transient stability analysis
methods, they consider simplified models formulated in reference or
relative coordinates (with uniform damping assumption) and result
mostly in numerical procedures rather than in concise
and simple conditions. For lossy power networks the cited articles consider
either special benchmark problems or networks with ``sufficiently small''
transfer conductances. To the best of our knowledge there are no results
quantifying this smallness of $\Re(Y_{ij})$ or $\varphi_{ij}$ for arbitrary
networks. Moreover, from a network perspective the existing methods do not
result in explicit and concise conditions relating synchronization to the
network's state, parameters, and topology. The following sections will
address these questions quantitatively via purely algebraic tests.

%% file: Sections/MainResult.tex
\subsection{The Non-Uniform Kuramoto Model}\label{Subsection: The Non-Uniform Kuramoto Model}

As we have already mentioned, there is a striking similarity between the
power network model \eqref{eq: Classical model} and the Kuramoto model
\eqref{eq: Kuramoto system}. To study this similarity, we define the
\emph{non-uniform Kuramoto model} by
\begin{equation}
  	D_{i}\, \dot{\theta}_{i}
  	=
  	\omega_{i}  - \sum\nolimits_{j=1}^{n} P_{ij} \sin(\theta_{i} - \theta_{j} + \varphi_{ij})
  	, \quad
  	i \in \until{n}
  	\label{eq: Non-uniform Kuramoto model}
  	\,,
\end{equation}
where we assume that the parameters satisfy $D_{i}>0$, $\omega_{i}\in\mbb
R$, $P_{ij}>0$, and $\varphi_{ij}\in[0,\pi/2[$, for all $i,j\in\until{n}$,
$i\neq{j}$; by convention, $P_{ii}$ and $\varphi_{ii}$ are set to
zero. System \eqref{eq: Non-uniform Kuramoto model} may be regarded as a
generalization of the classic Kuramoto model \eqref{eq: Kuramoto system}
with multiple time-constants $D_{i}$ and non-homogeneous but symmetric
coupling terms $P_{ij}$ and phase shifts $\varphi_{ij}$.  The non-uniform
Kuramoto model \eqref{eq: Non-uniform Kuramoto model} will serve as a link
between the power network model \eqref{eq: Classical model}, the Kuramoto
model \eqref{eq: Kuramoto system}, and the consensus protocol\,\,\eqref{eq:
  Consensus protocol}.

 \begin{remark}\label{Remark: second-to-first}
   {\bf(Second-order systems and their first-order approximations:)} The
   non-uniform Kuramoto model \eqref{eq: Non-uniform Kuramoto model} can be
   seen as a long-time approximation of the second order system \eqref{eq:
     Classical model} for a small ``inertia over damping ratio''
   $M_{i}/D_{i}$.
   Note the analogy between the non-uniform Kuramoto model \eqref{eq:
     Non-uniform Kuramoto model} and the dimension-reduced gradient system
   \eqref{eq: Classical model - gradient system - reduced model} studied in
   classic transient stability analysis to approximate the stability
   properties of the second-order system \eqref{eq: Classical model -
     gradient system}
   ~\cite{HDC-FFW-PPV:88,HDC-CCC:95,PV-FFW-RLC:85}. 
   Both models are of first order, have the same right-hand side, and
   differ only in the time constants $D_{i}$. Thus, both models have the
   same equilibria with the same stability properties and with regions of
   attractions bounded by the same separatrices \cite[Theorems
   3.1-3.4]{HDC-FFW-PPV:88}.  The reduced system \eqref{eq: Classical model
     - gradient system - reduced model} is formulated as a gradient-system
   to study the stability of the equilibria of \eqref{eq: Classical model -
     gradient system - reduced model} (possibly in relative
   coordinates). The non-uniform Kuramoto model \eqref{eq: Non-uniform
     Kuramoto model}, on the other hand, can be directly used to study
   synchronization and reveals the underlying
   network\,\,structure.\,\oprocend
\end{remark}
    
\subsection{Main Synchronization Result}\label{Subsection: Main Synchronization Result}

\noindent We can now state our main result on the power network model \eqref{eq: Classical model} and the non-uniform Kuramoto\,model\,\,\eqref{eq: Non-uniform Kuramoto model}. 

\begin{theorem}
\label{Theorem: Main Synchronization Result}
{\bf(Main synchronization result)}
Consider the power network model \eqref{eq: Classical model} and the
non-uniform Kuramoto model \eqref{eq: Non-uniform Kuramoto model}. Assume that the minimal lossless coupling of any oscillator to the network is
larger than a critical value, i.e., 
\begin{equation}
	\subscr{\Gamma}{min}
	:=
	n \min_{i \neq j} \left\{ \frac{P_{ij}}{D_{i}} \cos(\varphi_{ij}) \right\}
	>
	\subscr{\Gamma}{critical}
	:=
	\frac{1}{\cos(\subscr{\varphi}{max})}
	\Bigl(
	\max_{i \neq j}\left|\frac{\omega_{i}}{D_{i}} - \frac{\omega_{j}}{D_{j}}\right| 
	+
	2
	\max_{i \in \until n}
	\sum\limits_{j=1}^{n}
	\frac{P_{ij}}{D_{i}} \sin(\varphi_{ij})
	\Bigr)
  	\label{eq: key-assumption}
  	\,.
\end{equation}
Accordingly, define $\subscr{\gamma}{min} \in [0,\pi/2 - \subscr{\varphi}{max}[$ and $\subscr{\gamma}{max} \in {]\pi/2,\pi]}$ as unique solutions to 
the equations $\sin(\subscr{\gamma}{min}) = \sin(\subscr{\gamma}{max}) = \cos(\subscr{\varphi}{max}) \, \subscr{\Gamma}{critical} / \subscr{\Gamma}{min}$.

\smallskip
\noindent For the non-uniform Kuramoto model, 
\begin{enumerate}

\item {\bf phase cohesiveness}: the set $\bar\Delta(\gamma)$ is positively
  invariant for every $\gamma \in
  [\subscr{\gamma}{min},\subscr{\gamma}{max}]$, and each trajectory
  starting in $\Delta(\subscr{\gamma}{max})$ reaches
  $\bar\Delta(\subscr{\gamma}{min})$; and
  
\item {\bf frequency synchronization:} for every $\theta(0) \in
  \Delta(\subscr{\gamma}{max})$, the frequencies $\dot{\theta}_{i}(t)$
  synchronize exponentially to some frequency $\dot{\theta}_{\infty} \in
  [\subscr{\dot{\theta}}{min}(0) , \subscr{\dot{\theta}}{max}(0)]$.
  
\end{enumerate}
\noindent For the power network model, for all $\theta(0) \in
\Delta(\subscr{\gamma}{max})$ and initial frequencies $\dot \theta_{i}(0)$,
\begin{enumerate}
  
\item[3)] {\bf approximation error:} there exists a constant
  $\epsilon^{*}>0$ such that, if $\epsilon := \subscr{M}{max} / \subscr{D}{min} < \epsilon^{*}$, then the solution
  $(\theta(t),\dot \theta(t))$ of \eqref{eq: Classical model} exists for
  all $t \geq 0$ and   it holds uniformly in $t$ that
  \begin{equation}
    \label{eq: approx-errors}
    \begin{split}
    \bigr(\theta_{i}(t) -  \theta_{n}(t) \bigl) &= \bigl(\bar \theta_{i}(t) - \bar \theta_{n}(t) \bigr)
	+
	\mc O(\epsilon),
	\quad \forall t \geq 0, \; i \in \until{n-1},
	\\
	\dot{\theta}(t) &=  D^{-1} Q(\bar\theta(t)) 
	+
	\mc O(\epsilon),
	\quad\quad \forall t > 0
	\,,      
      \end{split}
    \end{equation}
    where $\bar \theta(t)$ is the solution to the non-uniform Kuramoto
    model \eqref{eq: Non-uniform Kuramoto model} with initial condition
    $\bar \theta(0) = \theta(0)$, and $D^{-1} Q(\bar \theta)$ is the power
    flow \eqref{eq: Classical model - power flow} scaled by the inverse
    damping $D^{-1}$;\,\,and

  \item[4)] {\bf asymptotic approximation error:} there exists $\epsilon$ and
    $\subscr{\varphi}{max}$ sufficiently small, such that the $\mc
    O(\epsilon)$ approximation errors in equation~\eqref{eq: approx-errors}
    converge to zero as $t \to \infty$.
    
\end{enumerate}
\end{theorem}

The proof of Theorem \ref{Theorem: Main Synchronization Result} is based on
a singular perturbation analysis of the power network model \eqref{eq:
  Classical model} (see Section \ref{Section: Singular Perturbation Analysis
  of Synchronization}) and a synchronization analysis of the non-uniform
Kuramoto model \eqref{eq: Non-uniform Kuramoto model} (see Section
\ref{Section: Synchronization of Non-uniform Kuramoto System}) and will be
postponed to the end of Section \ref{Section: Synchronization of
  Non-uniform Kuramoto System}. We discuss the assumption that the {\it
  perturbation parameter} $\epsilon$ needs to be small separately in the
next subsection and state the following remarks to Theorem \ref{Theorem:
  Main Synchronization Result}:

\begin{remark}[Physical interpretation of Theorem \ref{Theorem: Main Synchronization Result}:]
\label{Remark: Physical interpretation of Main Result}
The right-hand side of condition \eqref{eq: key-assumption} states the
worst-case non-uniformity in natural frequencies (the difference in
effective power inputs at each generator) and the worst-case lossy coupling
of a generator to the network ($P_{ij} \sin(\varphi_{ij}) = E_{i}
E_{j} \Re(Y_{ij})$ reflects the transfer conductance), both of which are
scaled with the rates $D_{i}$. The term $\cos(\subscr{\varphi}{max}) =
\sin(\pi/2-\subscr{\varphi}{max})$ corresponds to phase cohesiveness in $\Delta(\pi/2 - \subscr{\varphi}{max})$, which is necessary for the latter consensus-type
analysis.  These negative effects have to be dominated by the left-hand side of
\eqref{eq: key-assumption}, which is a lower bound for $\min_{i} \bigl\{ \sum_{j=1}^{n} \bigl(P_{ij} \cos(\varphi_{ij})/D_{i}\bigr) \bigr\}$, the worst-case
lossless coupling of a node to the network.  The multiplicative gap $\subscr{\Gamma}{critical} / \subscr{\Gamma}{min}$ between the right- and
the left-hand side in \eqref{eq: key-assumption} can be understood as a
robustness margin that additionally gives a {\it practical stability} result determining the admissible initial and the possible 
ultimate lack of phase cohesiveness in
$\bar\Delta(\subscr{\gamma}{min})$ and $\bar\Delta(\subscr{\gamma}{max})$.

In summary, the conditions of Theorem \ref{Theorem: Main Synchronization
  Result} read as ``the network connectivity has to dominate the network's
non-uniformity, the network's losses, and the lack of phase cohesiveness.''
In Theorem \ref{Theorem: Main Synchronization Result} we present 
the scalar synchronization condition \eqref{eq: key-assumption}, the
estimate for the region of attraction $\Delta(\subscr{\gamma}{max})$, and
the ultimate phase cohesive set $\bar\Delta(\subscr{\gamma}{min})$. In the
derivations leading to Theorem \ref{Theorem: Main Synchronization Result}
it is possible to trade off a tighter synchronization condition against a
looser estimate of the region of attraction, or a single loose scalar
condition against $n(n-1)/2$ tight pairwise conditions. These
  tradeoffs are explored in the appendix of this document. Finally, we
remark that the coupling weights $P_{ij}$ in condition \eqref{eq:
  key-assumption} are not only the reduced power flows but reflect for
uniform voltages $E_{i}$ and phase shifts $\varphi_{ij}$ also the
\emph{effective resistance} of the original (non-reduced) network topology
\cite{FD-FB:11d}.  Moreover, condition \eqref{eq:
  key-assumption} indicates at which generator the damping torque has to be
increased or decreased (via local power system stabilizers) in order to
meet the sufficient synchronization conditions.

The power network model \eqref{eq: Classical model} inherits the
synchronization condition \eqref{eq: key-assumption} in the relative
coordinates $\theta_{i} - \theta_{n}$ and up to the approximation error
\eqref{eq: approx-errors} which is of order $\epsilon$ and eventually
vanishes for $\epsilon$ and $\subscr{\varphi}{max}$ sufficiently small. The
relative coordinates can be shown to be well-posed (see Section
\ref{Section: Singular Perturbation Analysis of Synchronization}).
The convergence of the power network model only from almost all initial
conditions is a consequence of the existence of saddle points in the
non-uniform Kuramoto model. \oprocend
\end{remark}


\begin{remark}[Refinement of Theorem \ref{Theorem: Main Synchronization
    Result}:]
  Theorem \ref{Theorem: Main Synchronization Result} can also be stated
  for 
  two-norm bounds on the 
  parameters involving the algebraic connectivity (see Theorem
  \ref{Theorem: Synchronization Condition II}).  For a lossless network,
  explicit values for the synchronization frequency and the exponential
  synchronization rate as well as conditions for phase synchronization can
  be derived (see Theorems \ref{Theorem: Frequency synchronization} and
  \ref{Theorem: Phase Synchronization}).
%
  When specialized to the classic 
  Kuramoto model \eqref{eq: Kuramoto system}, the sufficient condition \eqref{eq:
    key-assumption} is improves the results
  \cite{FDS-DA:07,JLvH-WFW:93,NC-MWS:08,AJ-NM-MB:04,GSS-UM-FA:09}, and it can also shown to be a tight bound.  
  We refer the reader to the detailed comments in Section\,\,\ref{Section:
    Synchronization of Non-uniform Kuramoto System}.\oprocend
\end{remark}


\subsection{Discussion of the  Perturbation Assumption}\label{Subsection: Discussion of the epsilon-Assumption}

The assumption that each generator is strongly overdamped is captured by
the smallness of the perturbation parameter $\epsilon = \subscr{M}{max} /
 \subscr{D}{min}$. This choice of the perturbation parameter
and the subsequent singular perturbation analysis (in Section
\ref{Section: Singular Perturbation Analysis of Synchronization}) is
similar to the analysis of Josephson arrays \cite{KW-PC-SHS:98}, coupled
overdamped mechanical pendula \cite{RDL:08}, flocking models \cite{SYH-CL-BR-MS:11}, and also classic transient
stability analysis \cite[Theorem 5.2]{HDC-FFW-PPV:88},
\cite{DS-UR-BS-MP:01}. In the linear case, this analysis resembles the
well-known overdamped harmonic oscillator, which features one slow and one
fast eigenvalue. The overdamped harmonic oscillator exhibits two
time-scales and the fast eigenvalue corresponding to the frequency damping
can be neglected in the long-term phase dynamics. In the non-linear case
these two distinct time-scales are captured by a singular perturbation
analysis. In short, this reduction of a coupled-pendula system corresponds
to the assumption that damping to a synchronization manifold and
synchronization itself occur on separate time scales.

In the application to realistic generator models one has to be careful
under which operating conditions $\epsilon$ is indeed a small physical
quantity. Typically, $M_{i} \in [2 \textup s, 12 \textup s]/(2\pi f_{0})$ depending on
the type of generator and the mechanical damping (including damper winding
torques) is poor: $D_{i} \in [1,3]/(2 \pi f_{0})$. However, for the
synchronization problem also the generator's internal excitation control
have to be considered, which increases the {\it damping torque} to $D_{i}
\in [10,35]/(2 \pi f_{0})$ depending on the system load
\cite{PK:94,PMA-AAF:77,PWS-MAP:98}. In this case, $\epsilon \in \mc O(0.1)$
is indeed a small quantity and a singular perturbation approximation is
accurate. In fact, the recent power systems literature discusses the need for sufficiently large damping to enhance transient stability, see \cite{LFCA-NGB:00,CCC-HDC:10} and references therein.


We note that simulation studies show an accurate approximation of the power
network by the non-uniform Kuramoto model also for values of $\epsilon \in
\mc O(1)$, i.e., they indicate that the threshold $\epsilon^{*}$ may be
sizable. The theoretical reasoning is the topological equivalence discussed
in Subsection \ref{Subsection: Review of Classical Transient Stability
  Analysis} between the power network model \eqref{eq: Classical model} and
the first-order model \eqref{eq: Classical model - gradient system -
  reduced model}, which is again topologically equivalent to the
non-uniform Kuramoto model \eqref{eq: Non-uniform Kuramoto model}, as
discussed in Remark \ref{Remark: second-to-first}.  The synchronization condition \eqref{eq: key-assumption} on the non-uniform Kuramoto model \eqref{eq: Non-uniform Kuramoto model} guarantees exponential stability of the non-uniform Kuramoto dynamics formulated in relative coordinates $\theta_{i} - \theta_{n}$, which again implies local exponential stability of the power network model \eqref{eq: Classical model} in relative coordinates. These arguments are elaborated in detail in the next section.
Thus, from the viewpoint of
topological equivalence,  Theorem \ref{Theorem: Main Synchronization Result}
holds locally completely independent of $\epsilon>0$, and the magnitude of
$\epsilon$ gives a bound on the approximation errors \eqref{eq: approx-errors} during\,\,transients.

The analogies between the power network model \eqref{eq: Classical model}
and the reduced model \eqref{eq: Classical model - gradient system -
  reduced model}, corresponding to the non-uniform Kuramoto model
\eqref{eq: Non-uniform Kuramoto model}, are directly employed in the {\it
  PEBS} \cite{HDC-FFW-PPV:88} and {\it BCU} algorithms
\cite{HDC-CCC:95}. These algorithms are not only scholastic but applied by
the power industry \cite{H-DC:10}, which additionally supports the validity
of the approximation of the power network model by the non-uniform
Kuramoto\,\,model.



%% file: Sections/Singular_Perturbation_Analysis.tex

\subsection{Time-Scale Separation of the Power Network Model}\label{Subsection: Time-Scale Separation of the Power Network Model}

In this section, we put the approximation of the power network model
\eqref{eq: Classical model} by the non-uniform Kuramoto model \eqref{eq:
  Non-uniform Kuramoto model} on solid mathematical ground via a singular
perturbation analysis. The analysis by Tikhonov's method \cite{HKK:02}
requires a system evolving on Euclidean space and exponentially stable
fixed points. In order to satisfy the assumptions of Tikhonov's theorem, we
introduce two concepts. 

First, we introduce a smooth map from a suitable subset of $\mathbb{T}^n$
to a compact subset of $\real^{n-1}$. 
For $\gamma \in {[0,\pi[}$, define the map
$\map{\groundedphases}{\Delta(\gamma)} {\subscr{\Delta}{grnd}(\gamma) :=
  \setdef{\bar\delta\in\real^{n-1}}{|\bar\delta_{i}| < \gamma,\,
    \max_{i,j}|\bar\delta_i-\bar\delta_j|<\gamma,\, i,j \in \until{n-1}}}$ that
associates to the angles $(\theta_1,\dots,\theta_n) \in \Delta(\gamma)$ the
array of angle differences $\bar\delta$ with components
$\bar\delta_i=\theta_i-\theta_n$, for $i\in\until{n-1}$.  This map is well
defined, that is, $\bar\delta \in \subscr{\Delta}{grnd}(\gamma)$, because
$|\bar\delta_{i}| = |\theta_{i} - \theta_{n}| < \gamma$ for all $i \in \until{n-1}$ and
$|\bar\delta_{i} - \bar\delta_{j}| = |\theta_{i} - \theta_{j}| < \gamma$ for all
distinct $i,j \in \until{n-1}$.  Also, this map is smooth because $\gamma<\pi$ implies that all
angles take value in an open half-circle and their pairwise differences are
smooth functions (see Section~\ref{subsec:geom-torus}).  As a final remark,
note that the angle differences $\bar\delta_1,\dots,\bar\delta_{n-1}$ are
well-known in the transient stability \cite{HDC-CCC:95,JB:84} and in the
Kuramoto literature \cite{DA-JAR:04}, and we refer to them as
\emph{grounded angles} in the spirit of circuit theory. The sets
$\Delta(\pi)$ and $\subscr{\Delta}{grnd}(\pi)$ as well as the map $\theta
\mapsto \bar\delta = \textup{grnd}(\theta)$ are illustrated in
Figure\,\ref{Fig: grnd map}.
\begin{figure}[htbp]
	\centering{
	\includegraphics[scale=0.5]{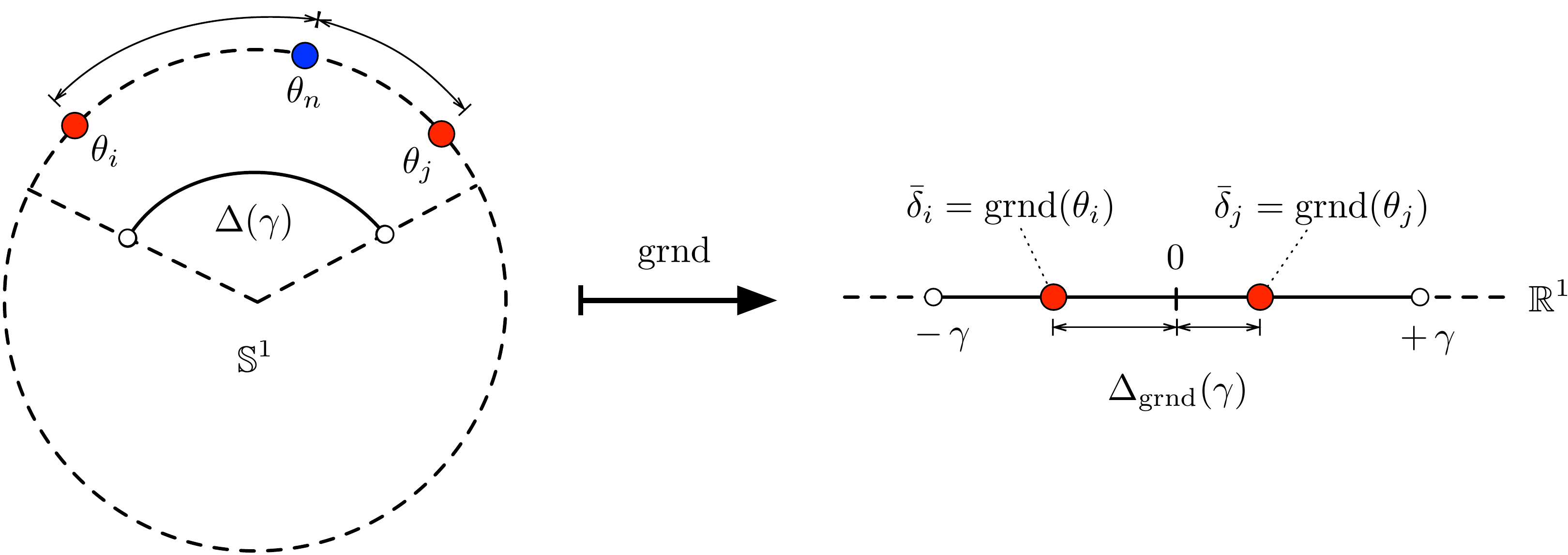}
	\caption{Illustration of the map $\textup{grnd}:\, \Delta(\gamma)
          \to \subscr{\Delta}{grnd}(\gamma)$. The map $\groundedphases$ can
          be thought of as as a symmetry-reducing projection from
          $\Delta(\gamma)$ (illustrated as subset of $\mbb S^{1}$) to
          $\subscr{\Delta}{grnd}(\gamma)$ (illustrated as subset of $\mbb
          R^{1}$), where $\theta_{n}$ is projected to the origin 0. The set
          $\Delta(\gamma)$ and the map $\textup{grnd}$ are invariant under
          translations on $\mbb T^{n}$ that is, under maps of the form
          $(\theta_1,\dots,\theta_n) \mapsto (\theta_1+\alpha, \dots
          \theta_n+\alpha)$.}
%
%
	\label{Fig: grnd map}
	}
\end{figure}

Second, by formally computing the difference between the angles $\dot{\theta}_i$ and
$\dot{\theta}_n$, we define \emph{grounded Kuramoto model} with state
${\delta}\in\real^{n-1}$ by
\begin{equation}
  \dot {{\delta}}_{i}
  =        
  \frac{\omega_{i}}{D_{i}} - \frac{\omega_{n}}{D_{n}}
  - \sum_{j=1, j\not=i}^{n-1} \!\!\Big(
  \frac{P_{ij}}{D_i} \sin(\delta_i-\delta_j+\varphi_{ij})
        + \frac{P_{nj}}{D_n}\sin(\delta_j-\varphi_{jn})
        \Big) - 
        \Big(
        \frac{P_{in}}{D_i}\sin(\delta_i+\varphi_{in}) + \frac{P_{in}}{D_n}\sin(\delta_i-\varphi_{in})
        \Big)
	\label{eq: grounded Kuramoto model}
	\,.
\end{equation}

\begin{lemma}[Properties of grounded Kuramoto model]
\label{Lemma: Properties of grounded Kuramoto model}
Let $\gamma \in {[0,\pi[}$ and let $\map{\theta}{\real_{\geq0}}{\mbb{T}^n}$
be a solution to the non-uniform Kuramoto model \eqref{eq: Non-uniform
  Kuramoto model} satisfying $\theta(0)\in\Delta(\gamma)$.
Let $\map{\delta}{\real_{\geq0}}{\real^{n-1}}$ be the solution to the
grounded Kuramoto model \eqref{eq: grounded Kuramoto model} with initial
condition $\delta(0)=\groundedphases(\theta(0)) \in
\subscr{\Delta}{grnd}(\gamma)$.
 
Then, $\delta(t)=\groundedphases(\theta(t))$ for
all $t\geq0$, if any one of the two following equivalent conditions holds:
\begin{enumerate}
  
\item {\bf phase cohesiveness:} the angles $\theta(t)$ take value in
  $\Delta(\gamma)$ for all time $t\geq0$; and
  
\item {\bf well-posedness:} the grounded angles $\delta(t)$ take value in
  $\subscr{\Delta}{grnd}(\gamma)$ for all time $t\geq0$.
\end{enumerate}
%
%
  Moreover, the following two statements are equivalent for any $\gamma \in {[0,\pi[}$:
  \begin{enumerate}\setcounter{enumi}{2}
  \item {\bf exponential frequency synchronization:} each trajectory of the non-uniform Kuramoto model satisfying the
    phase cohesiveness property 1) achieves exponential frequency
    synchronization; and
    
  \item {\bf exponential convergence to equilibria:} each trajectory of the grounded Kuramoto model satisfying the
    well-posedness property 2) converges exponentially to an equilibrium
    point.
  \end{enumerate}
Finally, every trajectory of the grounded Kuramoto model as in 4) satisfying the well-posedness property 2) with $\gamma \in {[0,\pi/2 - \subscr{\varphi}{max}]}$ converges to an exponentially stable equilibrium point.
\end{lemma}
\begin{proof}  
Note that, since both vector fields \eqref{eq: Non-uniform Kuramoto
      model} and \eqref{eq: grounded Kuramoto model} are locally Lipschitz,
    existence and uniqueness of the corresponding solutions follows provided that the corresponding
    evolutions are compact. Now, assume that 1) holds, that is, $\theta(t) \in \Delta(\gamma)$ (compact) for all $t
  \geq 0$.  Therefore, $\bar\delta(t)=\groundedphases(\theta(t)) \in
  \subscr{\Delta}{grnd}(\gamma)$ for all\,$t \geq 0$.  Also
  recall that the map $\groundedphases$ is smooth over $\Delta(\gamma)$.
  These facts and the definition of the grounded Kuramoto model~\eqref{eq:
    grounded Kuramoto model} imply that
  $\frac{d}{dt}\groundedphases(\theta(t))$ is well defined and identical to
  $\dot{{\delta}}(t)$ for all $t\geq0$. In turn, this implies that
  ${\delta}(t)=\groundedphases(\theta(t)) \in \subscr{\Delta}{grnd}(\gamma)$ holds for all positive
  times.
  
    Conversely, assume that 2) holds, that is, $\delta(t) \in
    \subscr{\Delta}{grnd}(\gamma)$ (compact) for all $t
  \geq 0$.
 %
    Due to existence and uniqueness and since initially $\delta(0)=\groundedphases(\theta(0))$ with $\theta(0) \in \Delta(\gamma)$, a set of angles $\theta(t) \in \Delta(\gamma)$ can be
    associated to $\delta(t) \in \subscr{\Delta}{grnd}(\gamma)$ such that $\delta(t) = \groundedphases(\theta(t))$ for all $t \geq 0$. 
 By construction of the grounded Kuramoto model~\eqref{eq:
    grounded Kuramoto model}, we have that $\theta(t)$ is identical to the solution to the non-uniform Kuramoto model \eqref{eq: Non-uniform Kuramoto model} for all $t \geq 0$. Thus, statement 2) implies statement 1) and $\delta(t) = \groundedphases(\theta(t))$ for all $t \geq 0$.
    Having established the equivalence of 1) and 2), we do not further distinguish between ${\delta}(t)$ and $\groundedphases(\theta(t))$.

  Assume that 3) holds, that is, all $\dot \theta_{i}(t)$ converge
  exponentially fast to some $\dot \theta_{\infty} \in \mbb R$. It follows
  that\,\,each $\dot \delta_{i}(t) = \dot \theta_{i}(t) - \dot
  \theta_{n}(t)$ converges exponentially fast to zero, and $\delta(t) =
  \delta(0) + \int_{0}^{t} \dot{\delta}(\tau) d\tau$ converges
  exponentially fast to some $\delta_{\infty} \in \subscr{\Delta}{grnd}(\gamma)$ due to property 2). Since the vector
  field \eqref{eq: grounded Kuramoto model} is continuous and
  $\lim\limits_{t \to \infty} \bigl( \delta(t), \dot \delta(t) \bigr) =
  (\delta_{\infty}, \fvec 0)$, the vector $\delta_{\infty}$ is necessarily
  an equilibrium of \eqref{eq: grounded Kuramoto model}, and property 4)\,\,follows.
  
  Assume that 4) holds, that is, all angular differences $\delta_{i}(t) =
  \theta_{i}(t) - \theta_{n}(t)$ converge exponentially fast to constant
  values $\delta_{i,\, \infty}$ for $i \in \until{n-1}$. This fact and the
  continuity of the vector field in equation~\eqref{eq: grounded Kuramoto
    model} imply that the array with entries $\delta_{i,\, \infty}$ is an
  equilibrium for \eqref{eq: grounded Kuramoto model} and that each
  frequency difference $\dot\delta_{i}(t) = \dot \theta_{i}(t) - \dot \theta_{n}(t)$ converges to zero. %
  Moreover, because the vector field in equation~\eqref{eq: grounded
    Kuramoto model} is analytic and the solution converges exponentially
  fast to an equilibrium point, the right-hand side of
  equation~\eqref{eq: grounded Kuramoto model} converges exponentially fast
  to zero and thus also the time-derivative of the solution, i.e., the
  array of all frequency differences, converges exponentially fast to zero.

To prove the final statement, assume that the non-uniform Kuramoto model \eqref{eq: Non-uniform Kuramoto model}  achieves frequency synchronization with synchronization frequency $\subscr{\dot\theta}{sync} \in \mbb R^{1}$ and phase cohesiveness in $\Delta(\pi/2 - \subscr{\varphi}{max})$. Thus, when formulated in a rotating coordinate frame with zero synchronization frequency, all trajectories $\theta_{i}(t) - \subscr{\dot\theta}{sync} \cdot t \pmod{2\pi}$ necessarily converge to an equilibrium $\theta^{*} \in \Delta(\pi/2 - \subscr{\varphi}{max})$. 

The Jacobian $J(\theta^{*})$ of the non-uniform Kuramoto model is given by the Laplacian matrix with weights $a_{ij}(\theta^{*}) = (P_{ij}/D_{i}) \cos(\theta^{*}_{i}-\theta^{*}_{j} + \varphi_{ij})$. For any $\theta^{*}\in \Delta(\pi/2 - \subscr{\varphi}{max})$, the weights $a_{ij}(\theta^{*})$ are strictly positive. In this case, it follows from the {\it contraction property} \cite[Theorem 1]{LM:04-arxiv} that the linearized dynamics $\dot \theta = J(\theta^{*}) \cdot \theta$ converge from any initial condition in $\mbb R^{n}$ to a point in the diagonal vector space $\fvec 1_{n}$. Hence, for any $\theta^{*}\in \Delta(\pi/2 - \subscr{\varphi}{max})$, the matrix $J(\theta^{*})$ has $n-1$ stable eigenvalues and one zero eigenvalue with eigenspace $\fvec 1_{n}$ corresponding to the translational invariance of the angular variable. 

Hence, any {\em equilibrium manifold} $\theta^{*} \in \Delta(\pi/2 - \subscr{\varphi}{max})$ (of dimension one due to translational invariance) is exponentially stable w.r.t. to the non-uniform Kuramoto dynamics \eqref{eq: Non-uniform Kuramoto model}. The corresponding {\em point} $\delta^{*} = \groundedphases(\theta^{*}(t)) \in \subscr{\Delta}{grnd}(\pi/2 - \subscr{\varphi}{max})$ (the translational symmetry is removed) is an equilibrium of the grounded Kuramoto dynamics \eqref{eq: grounded Kuramoto model} (due to property 4)). Finally, since  
$\theta^{*}$ is exponentially stable, it necessarily follows that $\delta^{*}$ is an exponentially stable equilibrium point.
\end{proof}

As mentioned in Remark \ref{Remark: second-to-first}, system \eqref{eq: Non-uniform Kuramoto model} may be
seen a long-time approximation of \eqref{eq: Classical model}, or spoken differently, it is the reduced system obtained by a singular perturbation
analysis. A physically reasonable singular perturbation parameter is the worst-case choice of $M_{i}/D_{i}$, that is,
$
\epsilon
:=
\subscr{M}{max}/ \subscr{D}{min}
$.
The dimension of $\epsilon$ is in seconds, which makes sense since time still has to be normalized with respect to $\epsilon$. If we reformulate the power network model \eqref{eq: Classical model} in grounded angular coordinates with the state $(\delta,\dot\theta) \in \mbb R^{n-1} \times \mbb R^{n}$, then we obtain the following system in singular perturbation standard form
\begin{align}
	\dt \delta_{i}
	=&
	f_{i}(\dot\theta)
	:=
	\dot \theta_{i} - \dot \theta_{n}
	\,,
	\quad i \in \until{n-1}
	\,,
	\label{eq: Power Grid - SP - standard model 1} \\
	\epsilon \dt \dot \theta_{i}
	=&
	g_{i}(\delta,\dot \theta)
	:=
	- F_{i} \, \dot \theta_{i}
	+ \frac{F_{i}}{D_{i}}\, \Bigl( \omega_{i} - \sum\nolimits_{j=1}^{n} P_{ij} \sin(\delta_{i} - \delta_{j} + \varphi_{ij}) \Bigr)
	\,,
	\quad i \in \until n
	\,,
	\label{eq: Power Grid - SP - standard model 2}
\end{align}
where
$
F_{i}
:= 
(D_{i}/\subscr{D}{min})/(M_{i}/\subscr{M}{max})
$ and $\delta_{n}:=0$ in equation \eqref{eq: Power Grid - SP - standard model 2}.
For $\epsilon$ sufficiently small, the long-term dynamics of \eqref{eq: Power Grid - SP - standard model 1}-\eqref{eq: Power Grid - SP - standard model 2} can be approximated by the grounded Kuramoto model \eqref{eq: grounded Kuramoto model} and the power flow \eqref{eq: Classical model - power flow}, where the approximation error is of order $\epsilon$ and $F_{i}$ determines its convergence rate in the fast\,\,time-scale.

\begin{theorem}
\label{Theorem: singular perturbations}
{\bf (Singular Perturbation Approximation)}
Consider the power network model \eqref{eq: Classical model} written as the singular perturbation problem \eqref{eq: Power Grid - SP - standard model 1}-\eqref{eq: Power Grid - SP - standard model 2} with bounded initial conditions $(\delta(0),\dot\theta(0))$, and the grounded non-uniform Kuramoto model \eqref{eq: grounded Kuramoto model} with initial condition $\delta(0)$ and solution $\bar \delta(t)$. Assume that there exists an exponentially stable fixed point $\delta_{\infty}$ of \eqref{eq: grounded Kuramoto model} and $\delta(0)$ is in a compact subset $\Omega_{\delta}$ of its region of attraction. 
Then, for each\,\,$\Omega_{\delta}$

\begin{enumerate}

	\item there exists $\epsilon_{*}>0$ such that for all $\epsilon < \epsilon_{*}$, the singular perturbation problem \eqref{eq: Power Grid - SP - standard model 1}-\eqref{eq: Power Grid - SP - standard model 2} has a unique solution $(\delta(t,\epsilon),\dot{\theta}(t,\epsilon))$ for $t \geq 0$, and for all $t \geq 0$ it holds uniformly in $t$ that
\begin{equation}
	\delta(t,\epsilon) - \bar \delta(t) 
	= 
	\mc O(\epsilon), \quad \text{and} \quad
	\dot{\theta}(t,\epsilon) - h(\bar \delta(t)) - y(t/\epsilon) 
	= 
	\mc O(\epsilon)
	\label{eq: singular perturbation error 1}
	\,,
\end{equation}
where 
$
y_{i}(t/\epsilon)
:=
(\dot{\theta}_{i}(0)-h_{i}(\delta(0))) \, e^{-F_{i} t/\epsilon}
$
and 
$h_{i}(\delta) := Q_{i}(\delta) / D_{i}$
for $i \in \until n$.

	\item For any $t_{b} \!>\! 0 $, there exists $\epsilon^{*} \!\leq\! \epsilon_{*}$ such that for all $t \!\geq\! t_{b}$ and whenever $\epsilon \!<\! \epsilon^{*}$ it holds uniformly\,\,that
\begin{equation}
	\dot{\theta}(t,\epsilon) - h(\bar \delta(t)) = \mc O(\epsilon)
	\label{eq: singular perturbation error 2}
	\,.
\end{equation}

	\item Additionally, there exist $\epsilon$ and  $\subscr{\varphi}{max}$ sufficiently small such that the approximation errors \eqref{eq: singular perturbation error 1}-\eqref{eq: singular perturbation error 2} converge exponentially to zero as $t \to \infty$.
\end{enumerate}
\end{theorem}

\begin{proof}
To prove statements 1) and 2) of Theorem \ref{Theorem: singular perturbations} we will follow {\it Tikhonov's theorem} \cite[Theorem 11.2]{HKK:02} and show that the singularly perturbed system \eqref{eq: Power Grid - SP - standard model 1}-\eqref{eq: Power Grid - SP - standard model 2} satisfies all assumptions of \cite[Theorem 11.2]{HKK:02} when analyzing it on $\mbb R^{n-1} \times \mbb R^{n}$ and after translating the arising fixed point to the origin.

{\it Exponential stability of the reduced system:} The {\it quasi-steady-state} of \eqref{eq: Power Grid - SP - standard model 1}-\eqref{eq: Power Grid - SP - standard model 2} is obtained by solving $g_{i}(\delta,\dot\theta) = 0$ for $\dot \theta$, resulting in the unique (and thus isolated) root $\dot \theta_{i} = h_{i}(\delta) = Q_{i}(\delta)/D_{i}$, $i \in \until n$. The {\it reduced system} is obtained as $\dot \delta_{i} = f_{i}(h(\delta))=h_{i}(\delta) - h_{n}(\delta)$, $i \in \until{n-1}$, which is equivalent to the grounded non-uniform Kuramoto model \eqref{eq: grounded Kuramoto model}. The reduced system is smooth, evolves on $\mbb R^{n-1}$, and by assumption  its solution $\bar\delta(t)$ is bounded and converges exponentially to the stable fixed point $\delta_{\infty}$. Define the error coordinates $x_{i}(t) = \bar\delta_{i}(t) - \delta_{i,\infty}$, $i \in \until{n-1}$ and the resulting system $\dot x = f(h(x+\delta_{\infty}))$ with state in $\mbb R^{n-1}$ and initial condition $x(0) = \delta(0) - \delta_{\infty}$. By assumption, the solution $x(t)$ is bounded and converges exponentially to the stable fixed point at $x = \fvec 0$. 

{\it Exponential stability of the boundary layer system:} Consider the error coordinate $y_{i}=\dot \theta_{i}-h_{i}(\delta)$, which shifts the error made by the quasi-stationarity assumption $\dot\theta_{i}(t) \approx h_{i}(\delta(t))$ to the origin. After stretching time to the dimensionless variable $\tau = t/\epsilon$, the quasi-steady-state error is
\begin{equation}
	\frac{d}{d\,\tau} \, y_{i}
	=
	g_{i}(\delta,y+h(\delta)) - \epsilon \dx{h_{i}}{\delta}{}\, f(y+h(\delta))
	= 
	- F_{i} \, y_{i} - \epsilon \dx{h}{\delta}{}\, f_{i}(y+h(\delta))
	\label{eq: Power Grid - SP - quasi-steady-state dynamics}
\end{equation}
with $y_{i}(0) = \dot\theta_{i}(0)-h_{i}(\delta(0))$. By setting $\epsilon = 0$,
\eqref{eq: Power Grid - SP - quasi-steady-state dynamics} reduces to the
{\it boundary layer model}
\begin{equation}
	\frac{d}{d\,\tau} \, y_{i}
	=
	- F \,y_{i}
	\,,\quad
	y_{i}(0) = \dot \theta_{i}(0)-h_{i}(\delta(0))
	\label{eq: Power Grid - SP - boundary model}
	\,.
\end{equation}
The boundary layer model \eqref{eq: Power Grid - SP - boundary model} is globally exponentially stable 
with solution $y_{i}(t/\epsilon) = y_{i}(0) e^{-F_{i} t/\epsilon}$, where $y_{i}(0)$ is in a compact subset of the region of attraction of \eqref{eq: Power Grid - SP - boundary model} due to boundedness of $(\delta(0),\dot{\theta}(0))$.
  
In summary, the singularly perturbed system \eqref{eq: Power Grid - SP - standard model 1}-\eqref{eq: Power Grid - SP - standard model 2} is smooth on $\mbb R^{n-1} \times \mbb R^{n}$, and the origins of the reduced system (in error coordinates) $\dot x = f(h(x+\delta_{\infty}))$ and the boundary layer model \eqref{eq: Power Grid - SP - boundary model} are exponentially stable (where Lyapunov functions are readily existent by converse arguments \cite[Theorem 4.14]{HKK:02}). Thus, all assumptions of \cite[Theorem 11.2]{HKK:02} are satisfied and statements 1) and 2) follow.

To prove statement 3), note that $\bar \delta(t)$ converges to an exponentially stable equilibrium point $\delta_{\infty}$, and $(\delta(t,\epsilon),\dot \theta(t,\epsilon))$ converges to an $\mc O(\epsilon)$ neighborhood of $\bigl( \delta_{\infty},h(\bar \delta_{\infty}) \bigl)$, where $h(\bar \delta_{\infty}) = \fvec 0$. We now invoke classic topological equivalence arguments from the transient stability literature \cite{HDC-FFW-PPV:88,HDC-CCC:95}. Both the second order system \eqref{eq: Power Grid - SP - standard model 1}-\eqref{eq: Power Grid - SP - standard model 2} as well as the reduced system $\dot \delta = f(h(\delta))$ correspond to the perturbed Hamiltonian system (8)-(9) in \cite{HDC-CCC:95} and the perturbed gradient system (10) in \cite{HDC-CCC:95}, where the latter is considered in \cite{HDC-CCC:95} with $D_{i} = 1$ for all $i$. Consider for a moment the case when all $\varphi_{ij} = 0$. In this case, the reduced system has a locally exponentially stable fixed point $\delta_{\infty}$ (for any $D_{i} > 0$ due to \cite[Theorem 3.1]{HDC-FFW-PPV:88}), and by \cite[Theorem 5.1]{HDC-CCC:95} we conclude that $(\delta_{\infty},\fvec 0)$ is also a locally exponentially stable fixed point of the second order system \eqref{eq: Power Grid - SP - standard model 1}-\eqref{eq: Power Grid - SP - standard model 2}. Furthermore, due to structural stability \cite[Theorem 5.7, R1]{HDC-CCC:95}, this conclusion holds also for sufficiently small phase shifts $\varphi_{ij}$. Thus, for sufficiently small $\epsilon$ and $\subscr{\varphi}{max}$, the solution of \eqref{eq: Power Grid - SP - standard model 1}-\eqref{eq: Power Grid - SP - standard model 2} converges exponentially to $(\delta_{\infty},\fvec 0)$. Consequently, the approximation errors $\delta(t,\epsilon) - \bar \delta(t)$ and $\dot \theta(t,\epsilon)-h(\bar \delta)$ as well as the boundary layer error $y(t/\epsilon)$  converge exponentially to zero. 
\end{proof}

Theorem \ref{Theorem: singular perturbations} can be interpreted geometrically as follows. The frequency dynamics of system \eqref{eq: Classical model}\,\,happen on a fast time-scale and converge exponentially to a {\it slow manifold} which can be approximated to first order by the scaled power flow $D^{-1} Q(\theta)$.
On this slow manifold the long-term phase synchronization dynamics of system \eqref{eq: Classical model} are given by the non-uniform Kuramoto model \eqref{eq: Non-uniform Kuramoto model}.

%% file: Sections/FrequencyEntrainment.tex
This section combines and extends methods from the consensus and Kuramoto literature to analyze the
non-uniform Kuramoto model \eqref{eq: Non-uniform Kuramoto model}.  The role of the time constants $D_{i}$ and the phase shifts $\varphi_{ij}$ is immediately revealed when dividing by $D_{i}$ both hand sides of \eqref{eq: Non-uniform Kuramoto model} and expanding the right-hand side as
\begin{equation}
	\dot \theta_{i}
	=
  	\frac{\omega_{i}}{D_{i}}  - \sum_{j=1, j \neq i}^{n}\! \bigg(
  	\frac{P_{ij}}{D_{i}} \cos(\varphi_{ij}) \sin(\theta_{i} - \theta_{j})
  	+
  	\frac{P_{ij}}{D_{i}} \sin(\varphi_{ij}) \cos(\theta_{i} - \theta_{j})
  	\bigg)
	\label{eq: non-uniform Kuramoto model - expanded}
  	\,.
\end{equation}
The difficulties in the analysis of system \eqref{eq: Non-uniform Kuramoto model} are the phase shift-induced lossy coupling $(P_{ij} / D_{i}) \sin(\varphi_{ij})$ $\times \cos(\theta_{i} - \theta_{j})$ inhibiting synchronization and the non-symmetric coupling between an oscillator pair $\{i,j\}$ via $P_{ij}/D_{i}$ on the one hand and $P_{ij}/D_{j}$ on the other. Since the non-uniform Kuramoto model \eqref{eq: Non-uniform Kuramoto model} is
derived from the power network model \eqref{eq: Classical model}, the
underlying graph induced by $P$ is {\it complete}
and {\it symmetric}, i.e., except for the diagonal entries, the matrix $P$
is fully populated and symmetric. For the sake of generality, this section considers the non-uniform Kuramoto
model \eqref{eq: Non-uniform Kuramoto model} under the assumption that the
graph induced by $P$ is neither complete nor symmetric, that is, some
$P_{ij}$ are zero and $P \neq P^{T}$.

\subsection{Frequency Synchronization of Phase-Cohesive Oscillators}\label{Subsection: Exponential Synchronization of Phase-Cohesive Oscillators}

Under the assumption of cohesive phases, the classic Kuramoto model \eqref{eq: Kuramoto system} achieves frequency synchronization \cite[Theorem 3.1]{NC-MWS:08}, \cite[Corollary 11]{GSS-UM-FA:09}. An analogous result guarantees frequency synchronization of non-uniform Kuramoto oscillators \eqref{eq: Non-uniform Kuramoto model} whenever the graph induced by $P$ has a globally reachable node.

\begin{theorem}
\label{Theorem: Frequency synchronization}
{\bf(Frequency synchronization)}
Consider the non-uniform Kuramoto model \eqref{eq: Non-uniform Kuramoto
  model} where the graph induced by $P$ has a globally reachable
node. 
Assume that there exists $\gamma \in [0,\pi/2 - \subscr{\varphi}{max}[$ 
such that the (non-empty) set of bounded phase differences
$\bar\Delta(\gamma)$ is positively invariant. Then for every $\theta(0) \in \bar\Delta(\gamma)$,

\begin{enumerate}

	\item the frequencies $\dot{\theta}_{i}(t)$ synchronize exponentially to some
frequency $\dot{\theta}_{\infty} \in [\subscr{\dot{\theta}}{min}(0) ,
\subscr{\dot{\theta}}{max}(0)]$; and

	\item if $\subscr{\varphi}{max} = 0$ and $P=P^{T}$, then  
$\dot \theta_{\infty} = \Omega := \sum\nolimits_{i} \omega_{i} / \sum_{i} D_{i}$ and the exponential synchronization rate is no worse than
\begin{equation}
	\!\!\!
	\subscr{\lambda}{fe} 
	=\!
	- \lambda_{2}(L(P_{ij}))  \cos(\gamma) \cos(\angle(D\fvec 1,\fvec 1))^{2} \!/ \subscr{D}{max}
	\label{eq: rate lambda_fe}
	.
\end{equation}

\end{enumerate}

\end{theorem}  
  
  In the definition of the convergence rate $\subscr{\lambda}{fe}$ in \eqref{eq: rate lambda_fe}, the factor $\lambda_{2}(L(P_{ij}))$ is the algebraic connectivity of the graph induced by $P=P^{T}$, the factor $1 / \subscr{D}{max}$ is the slowest time constant of the non-uniform Kuramoto system \eqref{eq: Non-uniform Kuramoto model}, the proportionality $\subscr{\lambda}{fe} \sim \cos(\gamma)$ reflects the phase cohesiveness in $\bar\Delta(\gamma)$, and the proportionality $\subscr{\lambda}{fe} \sim \cos(\angle(D \fvec 1,\fvec 1))^{2}$ reflects the fact that the error coordinate $\dot{\theta} - \Omega \fvec 1$ is for non-uniform damping terms $D_{i}$ not orthogonal to the {\it agreement vector} $\Omega \fvec 1$. For non-zero phase shifts a small signal analysis of the non-uniform Kuramoto model \eqref{eq: non-uniform Kuramoto model - expanded} reveals that the natural frequency of each oscillator diminishes as $\omega_{i} - \sum_{j \neq i} P_{ij} \sin(\varphi_{ij})$. In this case, and for symmetric coupling $P=P^{T}$, the synchronization frequency $\dot{\theta}_{\infty}$ in statement 1) will be smaller than $\dot{\theta}_{\infty} = \Omega$ in statement 2). When specialized to the classic Kuramoto model \eqref{eq: Kuramoto system}, statement 2) of Theorem \ref{Theorem: Frequency synchronization} reduces\,to\,\cite[Theorem 3.1]{NC-MWS:08}. 

\begin{proof}[Proof of Theorem \ref{Theorem: Frequency synchronization}]
By differentiating the non-uniform Kuramoto model \eqref{eq: Non-uniform Kuramoto model} we obtain the following dynamical system describing the evolution of the frequencies
\begin{equation}
	\dt D_{i} \dot{\theta}_{i}
	=
	- \sum\nolimits_{j=1}^{n} P_{ij} \cos(\theta_{i}-\theta_{j}+\varphi_{ij})\,(\dot{\theta}_{i} - \dot{\theta}_{j})
	\label{eq: dot theta dynamics}
	\,.
\end{equation}
Given the matrix $P$, consider a directed weighted graph $\mc G$ induced by the matrix with elements 
$a_{ij} = (P_{ij}/D_{i}) \cos(\theta_{i}-\theta_{j} + \varphi_{ij})$. 
By assumption we have for every $\theta(0) \in \bar\Delta(\gamma)$ that $\theta(t) \in \bar\Delta(\gamma)$ for all $t \geq 0$. Consequently the weights $a_{ij}(t) = (P_{ij}/D_{i}) \cos(\theta_{i}(t)-\theta_{j}(t)+ \varphi_{ij})$ are non-degenerate, i.e., zero for $P_{ij}=0$ and strictly positive otherwise for all $t \geq 0$. Note also that system \eqref{eq: dot theta dynamics} evolves on the tangent space of $\mbb T^{n}$, that is, the Euclidean space $\mbb R^{n}$. Therefore, the dynamics \eqref{eq: dot theta dynamics} can be analyzed as a linear time-varying consensus protocol for the velocities $\dot{\theta}_{i}$ with state-dependent Laplacian matrix $L(a_{ij})$:
\begin{equation}
	\dt \dot{\theta}
	=
	- L(a_{ij}) \, \dot{\theta}
	\label{eq: LPV concensus system for dot theta}
	\,,
\end{equation}
We analyze $L(a_{ij})$ as if it was a just time-varying Laplacian matrix $L(a_{ij}(t))$. 
At each time instance the matrix $-L(a_{ij}(t))$ is Metzler with zero row
sums, and the weights $a_{ij}(t)$ are bounded continuous functions of time
that induce integrated over any non-zero time interval a graph with
non-degenerate weights and a globally reachable node. 
It follows from the {\it contraction property} \cite[Theorem 1]{LM:04-arxiv}
that $\dot{\theta}_{i}(t) \in [\subscr{\dot{\theta}}{min}(0) ,
\subscr{\dot{\theta}}{max}(0)]$ for all $t \geq 0$ and $\theta_{i}(t)$ converge
exponentially to $\dot{\theta}_{\infty}$. This proves statement 1).\footnote{We remark that in the case of smoothly time-varying natural frequencies $\omega_{i}(t)$ an additional term $\dot \omega(t)$ appears on the right-hand side of the frequency consensus dynamics \eqref{eq: LPV concensus system for dot theta}. If the natural frequencies are non-identical or not exponentially convergent to identical values, the oscillators clearly cannot achieve frequency synchronization and the proof of Theorem \ref{Theorem: Frequency synchronization}\,fails.}

In the case of zero shifts and symmetric coupling $P=P^{T}$ the frequency dynamics \eqref{eq: LPV concensus system for dot theta} can be reformulated as a {\it symmetric} time-varying consensus protocol with multiple rates $D$ as
\begin{equation}
	\dt D \dot{\theta}
	=
	- L(w_{ij}(t)) \, \dot{\theta}
	\label{eq: LPV concensus system for dot theta - trivial phase shifts}
	\,,
\end{equation}
where $L(w_{ij}(t))$ is a symmetric time-varying Laplacian 
corresponding to a connected graph with strictly positive weights $w_{ij}(t) = P_{ij} \cos(\theta_{i}-\theta_{j})$. It follows from statement 1) that the oscillators synchronize exponentially to some frequency $\dot{\theta}_{\infty}$. Since $L(w_{ij})$ is symmetric,
it holds that $ \fvec 1_{n}^{T} \dt D \dot{\theta} = 0 $, or equivalently, $
\sum_{i} D_{i} \ddot{\theta}_{i}(t)$ is a constant conserved quantity. If we
apply this argument again at $\dot{\theta}_{\infty} := \lim_{t \to \infty}
\dot{\theta}_{i}(t)$, then we have $ \sum_{i} D_{i} \dot{\theta}_{i}(t) =
\sum_{i} D_{i} \dot{\theta}_{\infty} $, or equivalently, the frequencies
synchronize exponentially to
$\dot{\theta}_{\infty}=\Omega$. 

In order to derive an explicit synchronization rate, consider the {\it weighted disagreement vector}
$
\delta
=
\dot{\theta} - \Omega \fvec 1_{n}
$, 
as an error coordinate satisfying 
$
\fvec 1_{n}^{T} D \delta 
=
\fvec 1_{n}^{T} D \dot{\theta} - \fvec 1_{n}^{T} D \Omega \fvec 1_{n}
= 
0
$,
that is, $\delta$ lives in the {\it weighted disagreement eigenspace} of
co-dimension $1_{n}$ and with normal vector $D\,\fvec 1_{n}$. Since $\Omega$ is constant\,\,and $\ker(L(w_{ij})) = \mathrm{span}(\fvec 1_{n})$, 
the {\it weighted disagreement dynamics} are obtained from \eqref{eq: LPV concensus system for dot theta - trivial phase shifts} in $\delta$-coordinates\,as
\begin{equation}
	\dt D \delta
	=
	- L(w_{ij}(t)) \, \delta
	\label{eq: LPV concensus system for dot delta-disagreement}
	\,.
\end{equation}
Consider the {\it weighted disagreement function} 
$\delta\!\mapsto\!\delta^{T}D\delta$ 
and its derivative along the dynamics \eqref{eq: LPV concensus system for dot delta-disagreement}\,given\,by
\begin{equation*}
	\dt
	\delta^{T}D\delta
	=
	- 2 \, \delta^{T} L(w_{ij}(t)) \delta
	\,.
\end{equation*}
Since $\delta^{T} D \fvec 1_{n} = 0$, it follows that $\delta \not \in
\mathrm{span}(\fvec 1_{n})$ and $\delta$ can be uniquely decomposed into
orthogonal components as $ \delta = (\fvec 1_{n}^{T} \delta/n) \, \fvec 1_{n} +
\delta_{\perp} $, where $\delta_{\perp}$ is the orthogonal projection of
$\delta$ on the subspace orthogonal to $\fvec 1_{n}$. By the Courant-Fischer
Theorem \cite{CDM:01}, the time derivative of the weighted disagreement
function can be upper-bounded (point-wise in time) with the algebraic connectivity
$\lambda_{2}(L(P_{ij}))$ as follows:
\begin{multline*}
	\dt \delta^{T}D\delta 
	= - 2 \, \delta_{\perp}^{T} L(w_{ij}(t)) \delta_{\perp}
	= - (H\delta_{\perp})^{T} \cdot \diag(P_{ij} \cos(\theta_{i}-\theta_{j})) \cdot (H\delta_{\perp}) \\
	\leq - \min\nolimits_{\{i,j\} \in \mc E}\{\cos(\theta_{i}-\theta_{j}) :\, \theta \in \bar\Delta(\gamma) \}  \cdot (H\delta_{\perp})^{T} \cdot \diag(P_{ij} ) \cdot (H\delta_{\perp}) 
	\leq - \lambda_{2}(L(P_{ij})) \cos(\gamma) \cdot \|\delta_{\perp}\|_{2}^{2}
	\,.
\end{multline*}
In the sequel, $\norm{\delta_{\perp}}$ will be bounded by
$\norm{\delta}$. In order to do so, let $\fvec 1_{\perp} =
(1/\norm{\delta_{\perp}}) \, \delta_{\perp}$ be the unit vector that
$\delta$ is projected on (in the subspace orthogonal to $\fvec 1_{n}$). The
norm of $\delta_{\perp}$ can be obtained as
$
\norm{\delta_{\perp}}
=
\norm{\delta^{T} \fvec 1_{\perp}}
=
\norm{\delta} \cos(\angle(\delta,\fvec 1_{\perp})) 
$.
The vectors $\delta$ and $\fvec 1_{\perp}$ each live on $(n-1)$-dimensional linear hyperplanes with normal vectors $D \fvec 1_{n}$ and $\fvec 1_{n}$, respectively, see Figure \ref{Fig: Qualitative illustration of the disagreement eigenspace and the subspace orthogonal to the agreement eigenspace} for an illustration.
\begin{figure}[t]
	\centering{
	\includegraphics[scale=0.42]{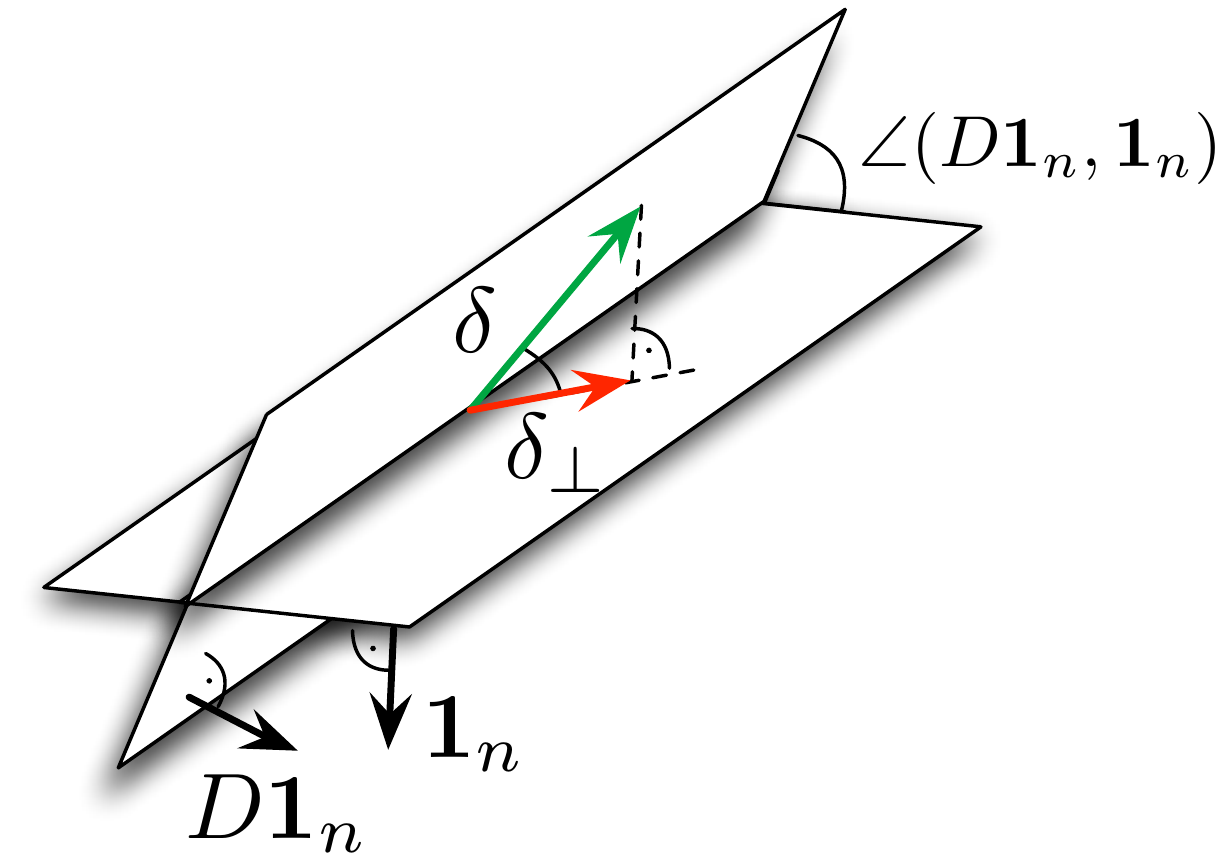}
	\caption{Illustration of the disagreement eigenspace and the orthogonal complement of $\mathbf 1_{n}$}
	\label{Fig: Qualitative illustration of the disagreement eigenspace and the subspace orthogonal to the agreement eigenspace}
	}
\end{figure}
The angle $\angle(\delta,\fvec 1_{\perp})$ is upper-bounded by
$\max_{\delta}\angle(\delta,\fvec 1_{\perp})$, which is said to be the {\it
  dihedral angle} and its sine is the {\it gap} between the two subspaces
\cite{CDM:01}.
Since both hyperplanes are of co-dimension~$1$, we obtain the dihedral
angle as the angle between the normal vectors $D \fvec 1_{n}$ and $\fvec 1_{n}$,
and it follows that $\angle(\delta,\fvec 1_{\perp}) \leq \angle(D \fvec
1_{n},\fvec 1_{n})$ (with equality for $n=2$). 
In summary, we have
$
\norm{\delta}
\geq
\norm{\delta_{\perp}}
\geq
\norm{\delta} \cos(\angle(D\fvec 1_{n},\fvec 1_{n})) 
$.

Finally, given $\subscr{D}{min} \norm{\delta}^{2} \leq \delta^{T}D\delta
\leq \subscr{D}{max} \norm{\delta}^{2}$ and $\subscr{\lambda}{fe}$ as stated in equation \eqref{eq: rate
  lambda_fe}, we obtain for the derivative of
the disagreement function
$
\dt
\delta^{T}D\delta
\leq
-
2\,\subscr{\lambda}{fe}
\delta^{T}D\delta
$.
An application of the Bellman-Gronwall Lemma 
yields $\delta(t)^{T}D\delta(t) \leq \delta(0)^{T}D\delta(0)\, e^{- 2
  \subscr{\lambda}{fe} (t)}$ for all $t \geq 0$. After
reusing the bounds on $\delta^{T}D\delta$, we obtain that the disagreement
vector $\delta(t)$ satisfies $ \norm{\delta(t)} \leq
\sqrt{\subscr{D}{max}/\subscr{D}{min}} \norm{\delta(0)} e^{-
  \subscr{\lambda}{fe} (t)} $ for all $t \geq 0$.
\end{proof}

%% file: Sections/PhaseClustering_1.tex
\subsection{Phase Cohesiveness}\label{Subsection: Phase Cohesiveness}

The key assumption in Theorem \ref{Theorem: Frequency synchronization} is that the angular distances are bounded in the set $\Delta(\pi/2 - \subscr{\varphi}{max})$. This subsection provides two different approaches to deriving conditions for this phase cohesiveness assumption - the contraction property and ultimate boundedness arguments. The dynamical system describing the evolution of the phase differences for the non-uniform Kuramoto model \eqref{eq: Non-uniform Kuramoto model} reads\,\,as
\begin{equation}
	\!
	\dot{\theta}_{i} - \dot{\theta}_{j}
	\!=\!
	\frac{\omega_{i}}{D_{i}} - \frac{\omega_{j}}{D_{j}}
	- \sum\nolimits_{k=1}^{n} \left(
	\frac{P_{ik}}{D_{i}} \sin(\theta_{i}-\theta_{k} + \varphi_{ik})
	 - \frac{P_{jk}}{D_{j}} \sin(\theta_{j}-\theta_{k} + \varphi_{jk}) 
	 \right)
	,\,
	 i,j \in \until n
	\label{eq: Non-uniform Kuramoto model - phase differences - in components} 
	.
\end{equation}
Note that equation \eqref{eq: Non-uniform Kuramoto model - phase differences - in components} cannot have a fixed point of the form $\dot \theta_{i} \!=\! \dot \theta_{j}$ if the following condition is not\,\,met.

\begin{lemma}[Necessary Condition on Synchronization]
\label{Lemma: Necessary sync condition}
Consider the non-uniform Kuramoto model \eqref{eq: Non-uniform Kuramoto
  model}. For any two distinct $i,j \in \until n$ there exists no solution
of the form $\dot \theta_{i}(t) = \dot \theta_{j}(t)$, $t \geq 0$,\,\,if
\begin{equation}
	\left| \frac{\omega_{i}}{D_{i}} - \frac{\omega_{j}}{D_{j}} \right|
	>
	\sum_{k=1}^{n} \Bigl( \frac{P_{ik}}{D_{i}} + \frac{P_{jk}}{D_{j}} \Bigr).
	\label{eq: necessary sync condition}
\end{equation}
\end{lemma}


Condition \eqref{eq: necessary sync condition} can be interpreted as ``the coupling between oscillators $i$ and $j$ needs to dominate their non-uniformity'' such that they can synchronize. For the classic Kuramoto model \eqref{eq: Kuramoto system} condition \eqref{eq: necessary sync condition} reduces to 
$
K < n/(2(n-1)) \cdot (\omega_{i} - \omega_{j})
$, a necessary condition derived also in \cite{NC-MWS:08,AJ-NM-MB:04,JLvH-WFW:93}. We\,\,remark that condition \eqref{eq: necessary sync condition} is only a loose bound for synchronization 
since it does take into account the effect of lossy coupling induced by the phase shift $\varphi_{ij}$, which becomes obvious when expanding the sinusoidal coupling terms in \eqref{eq: Non-uniform Kuramoto model - phase differences - in components} as in equation \eqref{eq: non-uniform Kuramoto model - expanded}. Nevertheless, condition \eqref{eq: necessary sync condition} indicates that the coupling needs to dominate the non-uniformity and possibly also disadvantageous effects of the lossy coupling.

In order to show the phase cohesiveness $\theta(t) \in \Delta(\pi/2 - \subscr{\varphi}{max})$, the Kuramoto
literature provides various methods such as quadratic Lyapunov functions
\cite{NC-MWS:08}, contraction mapping \cite{AJ-NM-MB:04}, geometric
\cite{FDS-DA:07}, or Hamiltonian arguments
\cite{REM-SHS:05,JLvH-WFW:93}.
 Due to the non-symmetric coupling via the weights $P_{ij}/D_{i}$ and the phase shifts
$\varphi_{ij}$ none of the mentioned methods appears to be easily
applicable to the non-uniform Kuramoto model non-uniform Kuramoto model  \eqref{eq: Non-uniform Kuramoto model}. 
A different approach from the
literature on consensus protocols \cite{LM:05,ZL-BF-MM:07,GSS-UM-FA:09} is
based on convexity and contraction and aims to show that the arc in which
all phases are contained is of non-increasing length. A modification of
this approach turns out to be applicable to non-uniform Kuramoto
oscillators with a complete coupling graph and results in the following theorem.

\begin{theorem}
\label{Theorem: Synchronization Condition I}
{\bf(Synchronization condition I)}
Consider the non-uniform Kuramoto-model \eqref{eq: Non-uniform Kuramoto
  model}, where the graph induced by $P=P^{T}$ is complete. Assume that the minimal lossless coupling of any oscillator to the network is
larger than a critical value, i.e., 
\begin{equation}
	\subscr{\Gamma}{min}
	:=
	n \min_{i \neq j} \left\{ \frac{P_{ij}}{D_{i}} \cos(\varphi_{ij}) \right\}
	>
	\subscr{\Gamma}{critical}
	:=
	\frac{1}{\cos(\subscr{\varphi}{max})}
	\Bigl(
	\max_{i \neq j}\left|\frac{\omega_{i}}{D_{i}} - \frac{\omega_{j}}{D_{j}}\right| 
	+
	2
	\max_{i \in \until n}
	\sum\limits_{j=1}^{n}
	\frac{P_{ij}}{D_{i}} \sin(\varphi_{ij})
	\Bigr)
  	\label{eq: key-assumption - Kuramoto}
  	\,.
\end{equation}
Accordingly, define $\subscr{\gamma}{min} \in [0,\pi/2 - \subscr{\varphi}{max}[$ and $\subscr{\gamma}{max} \in {]\pi/2,\pi]}$ as unique solutions to 
the equations $\sin(\subscr{\gamma}{min}) = \sin(\subscr{\gamma}{max}) = \cos(\subscr{\varphi}{max}) \, \subscr{\Gamma}{critical} / \subscr{\Gamma}{min}$. Then

\begin{enumerate}

%

\item {\bf phase cohesiveness}: the set $\bar\Delta(\gamma)$ is positively
  invariant for every $\gamma \in
  [\subscr{\gamma}{min},\subscr{\gamma}{max}]$, and each trajectory
  starting in $\Delta(\subscr{\gamma}{max})$ reaches
  $\bar\Delta(\subscr{\gamma}{min})$; and
  
\item {\bf frequency synchronization:} for every $\theta(0) \in
  \Delta(\subscr{\gamma}{max})$, the frequencies $\dot{\theta}_{i}(t)$
  synchronize exponentially to some frequency $\dot{\theta}_{\infty} \in
  [\subscr{\dot{\theta}}{min}(0) , \subscr{\dot{\theta}}{max}(0)]$.
    
\end{enumerate}
 
\end{theorem}

Condition \eqref{eq: key-assumption - Kuramoto} is interpreted in
Remark~\ref{Remark: Physical interpretation of Main Result}. 
In essence, Theorem \ref{Theorem: Synchronization Condition I} is based on
the {\it contraction property}: the positive invariance of $\bar\Delta(\gamma)$
is equivalent to showing that all angles $\theta_{i}(t)$ are contained in a rotating arc of non-increasing maximal length $\gamma$. This contraction analysis is similar to that of the
consensus algorithms 
in \cite{LM:05,ZL-BF-MM:07,GSS-UM-FA:09}, which derive their
results on $\mbb R^{n}$. Throughout the proof of Theorem \ref{Theorem: Synchronization Condition I} we comment on different possible branches leading to slightly different conditions than\,\,\eqref{eq: key-assumption - Kuramoto}. These branches are explored in detail in the Appendix \ref{Appendix}.

\begin{remark}[Reduction of Theorem \ref{Theorem: Synchronization Condition I} to classic Kuramoto oscillators:]\label{Remark: 1st Remark on
    Kuramoto bounds}
  For the classic Kuramoto oscillators \eqref{eq: Kuramoto system} the
  sufficient condition \eqref{eq: key-assumption - Kuramoto} of Theorem
  \ref{Theorem: Synchronization Condition I} specializes to
  \begin{equation}
    K
    >
    {\subscr{K}{critical}}
    :=
    \subscr{\omega}{max} - \subscr{\omega}{min}
    \label{eq: Kuramoto bound on coupling gain K}
    \,.
  \end{equation}
  In other words, if $K>{\subscr{K}{critical}}$, then 
for every $\theta(0) \in \Delta(\subscr{\gamma}{max})$ the oscillators synchronize and are ultimately phase cohesive in $\bar\Delta(\subscr{\gamma}{min})$, where $\subscr{\gamma}{max} \in {]\pi/2,\pi]}$ and $\subscr{\gamma}{min} \in {[0,\pi/2[}$ are the unique solutions to the equations  $\sin(\subscr{\gamma}{min}) = \sin(\subscr{\gamma}{max}) ={\subscr{K}{critical}}/K$. 
To the best of our knowledge, condition \eqref{eq: Kuramoto bound on
  coupling gain K} on the coupling gain $K$ is the tightest explicit
sufficient synchronization condition that has been presented in the
Kuramoto literature so far. In fact, the bound \eqref{eq: Kuramoto bound on
  coupling gain K} is close to the necessary condition for synchronization
$K> {\subscr{K}{critical}} \, n/(2(n-1))$ derived in Lemma \ref{Lemma:
  Necessary sync condition} and
\cite{NC-MWS:08,AJ-NM-MB:04,JLvH-WFW:93}. Obviously, for $n=2$ condition
\eqref{eq: Kuramoto bound on coupling gain K} is necessary and sufficient
for the onset of synchronization. Other sufficient bounds given in
the\,\,literature scale asymptotically with $n$, e.g., \cite[Theorem
2]{AJ-NM-MB:04} or \cite[proof of Theorem 4.1]{NC-MWS:08}. To compare our
condition \eqref{eq: Kuramoto bound on coupling gain K} with the bounds in
\cite{NC-MWS:08,FDS-DA:07,GSS-UM-FA:09}, we note from the proof of Theorem
\ref{Theorem: Synchronization Condition I} that our condition can be
equivalently stated as follows. The set 
$\bar\Delta(\pi/2-\gamma)$, for $\gamma \in {]0,\pi/2]}$, is positively
invariant if
  \begin{equation}
    K \geq K(\gamma) := \frac{\subscr{K}{critical}}{\cos(\gamma)} =
    \frac{\subscr{\omega}{max} - \subscr{\omega}{min}}{\cos(\gamma)} 
    \label{eq: Kuramoto bound on coupling gain K - 2}
    \,.
  \end{equation}
  Our bound \eqref{eq: Kuramoto bound on coupling gain K - 2} improves the
  bound $K>K(\gamma)n/2$ derived in \cite[proof of Theorem 4.1]{NC-MWS:08}
  via a quadratic Lyapunov function, the bound $K>K(\gamma)n/(n-2)$
  derived in \cite{GSS-UM-FA:09} via contraction arguments similar to ours, and
  the bound derived geometrically in \cite[proof of
  Proposition~1]{FDS-DA:07} that, after some manipulations, can be written
  in our notation as $K \geq K(\gamma)
  \cos((\pi/2-\gamma)/2)/\cos(\pi/2-\gamma)$.  Our ongoing research also reveals that the bound \eqref{eq: Kuramoto bound on coupling gain K} is
  tight for a bimodal distribution of the natural\,frequencies\,$\omega_{i} \!\in\! \{\subscr{\omega}{max},\subscr{\omega}{min}\}$ and also satisfies the
  implicit consistency conditions in \cite{DA-JAR:04}. Thus, \eqref{eq:
    Kuramoto bound on coupling gain K} is a necessary and sufficient
  condition for synchronization when the natural frequencies $\omega_{i}$
  are only known to be contained in
  $[\subscr{\omega}{min},\subscr{\omega}{max}]$. We elaborate on this interesting circle of ideas in a separate publication \cite{FD-FB:10w}.
  
    In summary, condition
  \eqref{eq: key-assumption - Kuramoto} in Theorem \ref{Theorem: Synchronization Condition I} improves the
  known~\cite{NC-MWS:08,AJ-NM-MB:04,JLvH-WFW:93,FDS-DA:07,GSS-UM-FA:09}
  sufficient conditions for synchronization of classic Kuramoto
  oscillators, and it is a necessary and sufficient if the particular distribution of the natural frequencies $\omega_{i} \in [\subscr{\omega}{min},\subscr{\omega}{max}]$ is unknown. 
  \oprocend
\end{remark}

{\itshape Proof of Theorem  \ref{Theorem: Synchronization Condition I}: }
We start by proving the positive invariance of $\bar\Delta(\gamma)$ for $\gamma \in {[0,\pi]}$.  Recall the geodesic distance between two angles on 
$\mathbb{T}^1$ and define the non-smooth function $\map{V}{\mathbb{T}^n}{[0,\pi]}$ by
\begin{equation*} 
  V(\psi) =  \max\setdef{|\psi_i-\psi_j|}{i,j\in\until{n}}.
\end{equation*}
By assumption, the angles $\theta_i(t)$ belong to the set
$\bar\Delta(\gamma)$ at time $t=0$, that is, they are all contained in an
arc of length $\gamma \in {[0,\pi]}$. In this case, $V(\psi)$ can
equivalently be written as maximum over a set of differentiable functions,
that is, $V(\psi) = \max\setdef{\psi_i-\psi_j}{i,j\in\until{n}}$. The arc
containing all angles has two boundary points: a counterclockwise maximum
and a counterclockwise minimum. If we let $\subscr{I}{max}(\psi)$
(respectively $\subscr{I}{min}(\psi)$) denote the set indices of the angles
$\psi_1,\dots,\psi_n$ that are equal to the counterclockwise maximum
(respectively the counterclockwise minimum), then we may\,\,write
\begin{equation*}
  V(\psi) =  \psi_{m'} - \psi_{\ell'},   \quad \text{for all } 
  m'\in\subscr{I}{max}(\psi)   \text{ and }
  \ell'\in\subscr{I}{min}(\psi).
\end{equation*}
We aim to show that all angles remain in $\bar\Delta(\gamma)$ for all subsequent times
$t>0$.  Note that $\theta(t)\in\bar\Delta(\gamma)$ if and only if
$V(\theta(t))\leq \gamma$.  Therefore,
$\bar\Delta(\gamma)$ is positively invariant if and only if $V(\theta(t))$ does
not increase at any time $t$ such that
$V(\theta(t))= \gamma$. 
The {\it upper Dini derivative} of
$V(\theta(t))$ along the dynamical system~\eqref{eq: Non-uniform Kuramoto
  model - phase differences - in components} is given as in \cite[Lemma
2.2]{ZL-BF-MM:07}
\begin{align*}
	D^{+} V (\theta(t))
	&=
	\lim_{h \downarrow 0}\sup \frac{V(\theta(t+h)) - V(\theta(t))}{h}
	=
	\dot{\theta}_{m}(t) - \dot{\theta}_{\ell}(t)
	\,,
\end{align*}
where $m \in \subscr{I}{max}(\theta(t))$ and $\ell \in
\subscr{I}{min}(\theta(t))$ are indices with the properties that
\begin{gather*}
  \dot{\theta}_m(t) = \max \setdef{ \dot{\theta}_{m'}(t) }{m'\in
    \subscr{I}{max}(\theta(t))}, 
  \enspace \text{and } \enspace 
  \dot{\theta}_\ell(t) = \min
  \setdef{ \dot{\theta}_{\ell'}(t) }{\ell'\in \subscr{I}{min}(\theta(t))}.
\end{gather*}
Written out in components (in the expanded form \eqref{eq: non-uniform Kuramoto model - expanded}) $D^{+} V (\theta(t))$ takes the form
\begin{align}
	D^{+} V (\theta(t))
	=&\;
	\frac{\omega_{m}}{D_{m}} - \frac{\omega_{\ell}}{D_{\ell}} - 
	\sum\nolimits_{k=1}^{n} \left( 
	a_{mk} \sin(\theta_{m}(t) - \theta_{k}(t)) 
	+
	a_{\ell k} \sin(\theta_{k}(t)-\theta_{\ell}(t)) 
	\right)
	\notag
	\\
	& - 
	\sum\nolimits_{k=1}^{n} \left( 
	b_{mk} \cos(\theta_{m}(t) - \theta_{k}(t)) 
	- 
	b_{\ell k} \cos(\theta_{\ell}(t)-\theta_{k}(t)) 
	\right)	
	\label{eq: D+ with cosine terms}
	\,,
\end{align}
where we used the abbreviations $a_{ik} := P_{i k} \cos(\varphi_{ik})/D_{i}$ and $b_{ik} := P_{ik}\sin(\varphi_{ik})/D_{i}$. The equality $V(\theta(t)) = \gamma$ implies that, measuring distances counterclockwise and modulo additional
terms equal to multiples of $2\pi$, we have $\theta_{m}(t) - \theta_{\ell}(t) = \gamma$, $0 \leq \theta_{m}(t) - \theta_{k}(t) \leq \gamma$,
and $0 \leq \theta_{k}(t) - \theta_{\ell}(t) \leq \gamma$. To simplify the notation in the subsequent arguments, we do not aim at the tightest and least conservative bounding of the two sums on the right-hand side of \eqref{eq: D+ with cosine terms} and continue as follows.\footnote{Besides tighter bounding of the right-hand side of \eqref{eq: D+ with cosine terms}, the proof can alternatively be continued by adding and subtracting\,the coupling with zero phase shifts in \eqref{eq: D+ with cosine terms} or by noting that the right-hand side of \eqref{eq: D+ with cosine terms} is a convex function of $\theta_{k} \in [\theta_{\ell},\theta_{m}]$ that achieves its maximum at the boundary $\theta_{k} \in \{\theta_{\ell},\theta_{m}\}$. If the analysis is restricted to $\gamma \in {[0,\pi/2]}$, the term $b_{mk}$ can be dropped.}

Since both sinusoidal terms on the right-hand side of \eqref{eq: D+ with cosine terms} are positive, they can be lower-bounded as
\begin{multline*} 
	a_{mk} \sin(\theta_{m}(t) - \theta_{k}(t)) 
	+
	a_{\ell k} \sin(\theta_{k}(t)-\theta_{\ell}(t))
	\geq
	\min_{i \in \{m,\ell\} \setminus \{k\}} \left\{ a_{ik}  \right\}
	\big( \sin(\theta_{m}(t) - \theta_{k}(t)) + \sin(\theta_{k}(t)-\theta_{\ell}(t)) \big)
	\\
	=
	2\, \min_{i \in \{m,\ell\} \setminus \{k\}} \left\{ a_{ik}  \right\}
	\sin\!\left(\frac{\theta_{m}(t) - \theta_{\ell}(t)}{2}\right) 
	\cos\!\left(\frac{\theta_{m}(t) + \theta_{\ell}(t)}{2} - \theta_{k}(t) \right)
	\\
	\geq
	2 \, \min_{i \in \{m,\ell\} \setminus \{k\}} \left\{ a_{ik}  \right\}
	\sin\Bigl( \frac{\gamma}{2} \Bigr) \cos\Bigl( \frac{\gamma}{2} \Bigr)
	=
	\min_{i \in \{m,\ell\} \setminus \{k\}} \left\{ a_{ik}  \right\}	\sin(\gamma)
	\,,
\end{multline*}
where we applied the trigonometric identities $\sin(x) + \sin(y) = 2 \sin(\frac{x+y}{2}) \cos(\frac{x-y}{2})$ and $2 \sin(x) \cos(y) = \sin(x-y) + \sin(x+y)$. The cosine terms in \eqref{eq: D+ with cosine terms} can be lower bounded in $\bar\Delta(\gamma)$  as
$
	b_{mk} \cos(\theta_{m}(t) - \theta_{k}(t)) 
	- 
	b_{\ell k} \cos(\theta_{\ell}(t)-\theta_{k}(t))
	\geq
	- b_{mk} - b_{\ell k} 
$. In summary, $D^{+}V(\theta(t))$ in \eqref{eq: D+ with cosine terms} can be upper bounded by 
\begin{align*}
	D^{+} V (\theta(t)) 
	\leq&\;
	\frac{\omega_{m}}{D_{m}} - \frac{\omega_{\ell}}{D_{\ell}} -
  	\sum\nolimits_{k=1}^{n}  
	\min_{i \in \{m,\ell\} \setminus \{k\}} \left\{ a_{ik}  \right\} \sin(\gamma)
	+
	\sum\nolimits_{k=1}^{n}
	b_{m k}
	+
	\sum\nolimits_{k=1}^{n}
	b_{\ell k}
	\\
	\leq&\;
	\max_{i \neq j}\left|\frac{\omega_{i}}{D_{i}} - \frac{\omega_{j}}{D_{j}}\right| -
  	n \min_{i \neq j} \left\{ \frac{P_{ij}}{D_{i}} \cos(\varphi_{ij}) \right\} \sin(\gamma)
	+
	2\,
	\max_{i \in \until n}
	\sum\nolimits_{j=1}^{n}
	b_{ij}
	\,,
\end{align*}
where we further maximized the coupling terms and the differences in natural frequencies over all possible pairs $\{m,\ell\}$.
It follows that $V(\theta(t))$ is non-increasing for all $\theta(t) \in \bar\Delta(\gamma)$ if
\begin{equation}
	\subscr{\Gamma}{min} \sin(\gamma)
	\geq
	\cos(\subscr{\varphi}{max}) \, \subscr{\Gamma}{critical}
	\label{eq: D+ inequality}
	\,,
\end{equation}
where $\subscr{\Gamma}{min}$ and $\subscr{\Gamma}{critical}$ are defined in \eqref{eq: key-assumption - Kuramoto}.
The left-hand side of \eqref{eq: D+ inequality} is a strictly concave function of $\gamma \in {[0,\pi]}$. Thus, there exists an open set of arc lengths $\gamma$ including $\gamma^{*} = \pi/2-\subscr{\varphi}{max}$ satisfying inequality \eqref{eq: D+ inequality} if and only if inequality \eqref{eq: D+ inequality} is true at $\gamma^{*}=\pi/2-\subscr{\varphi}{max}$ with the strict inequality\,\,sign%
, which corresponds to condition \eqref{eq: key-assumption - Kuramoto} in the statement of Theorem \ref{Theorem: Synchronization Condition I}. Additionally, if these two equivalent statements are true, then $V(\theta(t))$ is non-increasing in $\bar\Delta(\gamma)$ for all $\gamma \in [\subscr{\gamma}{min},\subscr{\gamma}{max}]$, where $\subscr{\gamma}{min} \in {[0,\pi/2 - \subscr{\varphi}{max}[}$ and $\subscr{\gamma}{max} \in {]\pi/2,\pi]}$ are given as unique solutions to inequality \eqref{eq: D+ inequality} with equality sign. Moreover, $V(\theta(t))$ is strictly decreasing in $\bar\Delta(\gamma)$ for all $\gamma \in {]\subscr{\gamma}{min},\subscr{\gamma}{max}[}$. This concludes the proof of statement 1) and ensures that for every $\theta(0) \in \Delta(\subscr{\gamma}{max})$, there exists $T \geq 0$ such that $\theta(t) \in \bar\Delta(\pi/2 - \subscr{\varphi}{max})$ for all $t \geq T$. Thus, the positive invariance assumption of Theorem \ref{Theorem: Frequency synchronization} is satisfied, and statement 2) of Theorem \ref{Theorem: Synchronization Condition I} follows.
\hspace*{\fill}~\QED\par\endtrivlist\unskip

%% file: Sections/PhaseClustering_2.tex
In summary, Theorem \ref{Theorem: Synchronization Condition I} presents
sufficient conditions for the synchronization of the non-uniform Kuramoto
model and is based on the bound \eqref{eq:
  key-assumption - Kuramoto}. Condition \eqref{eq: key-assumption -
  Kuramoto} is a worst-case 
  bound, both on the parameters and on the initial angles. In the remainder of this section, we aim at deriving a two-norm type bound and require only connectivity of the graph induced by $P=P^{T}$ and not necessarily completeness.

We start our discussion with some preliminary notation and concepts. The following analysis is formally carried out for the complete graph, but, without loss of generality, we assume that some weights $P_{ij}=P_{ji}$ can be zero and the non-zero weights $P=P^{T}$ induce a connected graph.
Let $H \in \mbb R^{n(n-1)/2 \times n}$ be the incidence matrix of the complete graph with $n$ nodes and recall that for a vector
$x \in \mbb R^{n}$ the vector of {\it all} difference variables is $H x = (x_{2} - x_{1}, \dots)$.
The phase difference dynamics \eqref{eq: Non-uniform Kuramoto model - phase differences - in components} (with the sinusoidal coupling expanded as in \eqref{eq: non-uniform Kuramoto model - expanded}) can be reformulated in a compact vector notation\,\,as
\begin{equation}
	\dt H \theta
	=
	H D^{-1} \omega  
	- HD^{-1}H^{T}\diag(P_{ij}\cos(\varphi_{ij}))\sinbf(H\theta)
	- H X
	\label{eq: Non-uniform Kuramoto model - phase differences - in vector form}
	\,,
\end{equation}
where $X 
\in \mbb R^{n}$ is the vector of lossy coupling with components
$X_{i} = \sum\nolimits_{j=1}^{n} (P_{ij}/D_{i}) \sin(\varphi_{ij}) \cos(\theta_{i}-\theta_{j})$ and $\sinbf(H\theta)$ is the multivariable sine.
%
  The set of differential equations \eqref{eq: Non-uniform Kuramoto model - phase differences - in vector form} is well defined on $\mbb T^{n}$:
the left-hand side of \eqref{eq: Non-uniform Kuramoto model - phase differences - in vector form} is the vector of frequency differences $H \dot{\theta}=(\dot{\theta}_{2} - \dot{\theta}_{1},\dots)$ taking values in the tangent space to $\mbb T^{n}$, and 
the right-hand side of \eqref{eq: Non-uniform Kuramoto model - phase differences - in vector form} is a well-posed vector-valued function of $\theta \in \mbb T^{n}$.  

With slight abuse of notation, we denote the two-norm of the vector of all
 geodesic distances by $\norm{H \theta}_{2} =( \sum_{i} \sum_{j}
|\theta_{i}-\theta_{j}|^{2})^{1/2}$, and aim at ultimately bounding the
evolution of $\norm{H \theta(t)}_{2}$. Following a classic Kuramoto
analysis \cite{AJ-NM-MB:04,EC-PM:08,REM-SHS:05,JLvH-WFW:93}, we note that
the non-uniform Kuramoto model \eqref{eq: Non-uniform Kuramoto model} with
$\omega \equiv \fvec 0$ constitutes a Hamiltonian system with the
Hamiltonian $U(\theta)|_{\omega=\fvec 0}$ defined in
equation \eqref{eq: potential energy}. An analysis of \eqref{eq: Non-uniform Kuramoto model - phase differences -
  in vector form} by Hamiltonian arguments is possible, but results in very
conservative conditions. In the recent Kuramoto
literature \cite{NC-MWS:08,AJ-NM-MB:04}, a different Lyapunov function
considered for the uniform Kuramoto model \eqref{eq: Kuramoto system} evolving on $\mbb R^{n}$ is
simply $\norm{H\theta}_{2}^{2}$. Unfortunately, in the case of non-uniform
rates $D_{i}$ this function's Lie derivative is sign-indefinite.
However, it is possible to identify a similar Lyapunov function
that has a Lie derivative with symmetric coupling. Consider the function $\mc W:\,\mbb T^{n} \to \mbb R$ defined\,\,by
\begin{equation}
	\mc W(\theta)
	=
	\frac{1}{4} \sum\nolimits_{i=1}^{n} \sum\nolimits_{j=1}^{n} D_{i}D_{j}\, |\theta_{i}-\theta_{j}|^{2}
	\label{eq: mathcal W(theta)}
	\,.
\end{equation}
A Lyapunov analysis of system \eqref{eq: Non-uniform Kuramoto model - phase
  differences - in vector form} via the Lyapunov function $\mc W$ leads to
the following theorem.

\begin{theorem}
\label{Theorem: Synchronization Condition II}
{\bf(Synchronization condition II)}
Consider the non-uniform Kuramoto model \eqref{eq: Non-uniform Kuramoto model},
where the graph induced by $P=P^{T}$ is connected. Let $H \in \mbb
R^{n(n-1)/2 \times n}$ be the incidence matrix of the\,\,complete graph and
assume that the algebraic connectivity of the lossless coupling is larger
than a critical\,value,\,i.e.,
\begin{equation}
  \label{eq: key-assumption - Kuramoto - 2}
  \lambda_{2}(L(P_{ij}\cos(\varphi_{ij}))) > \subscr{\lambda}{critical} := 
  \frac{\norm{H D^{-1} \omega}_{2} 
  + 
  \sqrt{n} \, \Big|\Big|\Big[ \sum\nolimits_{j=1}^{n}
  \frac{P_{1j}}{D_{1}}\sin(\varphi_{1j}), \dots,
  \sum\nolimits_{j=1}^{n} \frac{P_{nj}}{D_{n}}\sin(\varphi_{nj})
  \Big]\Big|\Big| _{2} }
  {\cos(\subscr{\varphi}{max}) (\kappa/n) \alpha/\max\nolimits_{i\neq j} \{ D_{i} D_{j} \}}
  \,,
\end{equation}
where $\kappa := \sum_{k=1}^{n} D_{k}$ and $\alpha := \sqrt{
  \min\nolimits_{i \neq j} \{D_{i} D_{j} \} / \max\nolimits_{i \neq j}
  \{D_{i} D_{j} \} }$.

Accordingly, define $\subscr{\gamma}{max} \in
{]\pi/2-\subscr{\varphi}{max},\pi]}$ and $\subscr{\gamma}{min} \in
{[0,\pi/2-\subscr{\varphi}{max}[}$ to be the unique solutions to the
equations $\sinc(\subscr{\gamma}{max})/\sinc(\pi/2-\subscr{\varphi}{max}) =
\sin(\subscr{\gamma}{min}) / \cos(\subscr{\varphi}{max}) =
\subscr{\lambda}{critical}/\lambda_{2}(L(P_{ij} \cos(\varphi_{ij})))$. Then,

%

\begin{enumerate}	
\item {\bf phase cohesiveness:} the set $\{\theta \in \Delta(\pi):\, \norm{H\theta}_{2} \leq \gamma \}$ is positively invariant for every $\gamma \in [\subscr{\gamma}{min},\alpha\subscr{\gamma}{max}]$, and each trajectory starting in $\{\theta \in \Delta(\pi):\, \norm{H\theta(0)}_{2} < \alpha \subscr{\gamma}{max}\}$ reaches $\{\theta \in \Delta(\pi):\,\norm{H\theta}_{2} \leq \subscr{\gamma}{min} \}$; and
  
\item {\bf frequency synchronization:} for every $\theta(0) \in \Delta(\pi)$ with $\norm{H\theta(0)}_{2}
< \alpha \subscr{\gamma}{max}$ the frequencies $\dot{\theta}_{i}(t)$
  synchronize exponentially to some frequency $\dot{\theta}_{\infty} \in
  [\subscr{\dot{\theta}}{min}(0) ,
  \subscr{\dot{\theta}}{max}(0)]$. Moreover, if $\varphi_{\max} = 0$, then
  $\dot{\theta}_{\infty} = \Omega$ and the exponential synchronization rate
  is no worse than $\subscr{\lambda}{fe}$ as defined in equation \eqref{eq:
    rate lambda_fe}.
\end{enumerate}
\end{theorem}

\begin{remark}{\bf(Physical interpretation of Theorem \ref{Theorem: Synchronization Condition II}:)}
In condition \eqref{eq: key-assumption - Kuramoto - 2}, $\big|\!\big|\!\big[ \!\dots, \sum\nolimits_{j=1}^{n} \frac{P_{ij}}{D_{i}}\sin(\varphi_{ij}), \dots\! \big]\!\big|\!\big| _{2}$ is the two-norm of the vector containing the lossy coupling, $\norm{H D^{-1} \omega}_{2} \!=\! \norm{(\omega_{2}/D_{2} - \omega_{1}/D_{1},\dots)}_{2}$ corresponds to the non-uniformity in the natural frequencies, $\lambda_{2}(L(P_{ij} \cos(\varphi_{ij})))$ is the algebraic connectivity induced by the lossless coupling, $\cos(\subscr{\varphi}{max}) = \sin(\pi/2 - \subscr{\varphi}{max})$ reflects again the phase cohsiveness in $\Delta(\pi/2 - \subscr{\varphi}{max})$, and $(\kappa/n) \alpha/\max\nolimits_{i\neq j} \{ D_{i} D_{j} \}$ weights the non-uniformity in the time constants\,\,$D_{i}$. The gap in condition \eqref{eq: key-assumption - Kuramoto - 2} yields a again practical stability result determining the initial and ultimate phase cohesiveness. 
Condition \eqref{eq: key-assumption - Kuramoto - 2} can be extended to non-reduced power network models \cite{FD-FB:11d}.
\oprocend
\end{remark}

\begin{remark}
\label{Remark: 2nd Remark on Kuramoto bounds}
{\bf (Reduction of Theorem \ref{Theorem: Synchronization Condition II} to classic Kuramoto oscillators:)}
  For classic Kuramoto oscillators \eqref{eq: Kuramoto system}, condition \eqref{eq: key-assumption - Kuramoto - 2} reduces to 
  $K > \subscr{K}{critical}^{*} := \norm{H\omega}_{2}$, which is a more conservative bound\,\,than
  $K > \subscr{K}{critical} = \subscr{\omega}{max} - \subscr{\omega}{min}$
  presented in \eqref{eq: Kuramoto bound on coupling gain K}.
  It follows that the oscillators synchronize for $\norm{H \theta(0)}_{2} < \subscr{\gamma}{max}$ and are ultimately phase cohesive in $\norm{H\theta}_{2} \leq\subscr{\gamma}{min}$, where $\subscr{\gamma}{max} \in {]\pi/2,\pi]}$ and $\subscr{\gamma}{min} \in {[0,\pi/2[}$ are the unique solutions to 
  $(\pi/2) \sinc(\subscr{\gamma}{max}) = \sin(\subscr{\gamma}{min})  =  \subscr{K}{critical}^{*}/K$.
   The Lyapunov function $\mc W(\theta)$ reduces to the one used in \cite{NC-MWS:08,AJ-NM-MB:04}
and can also be used to prove \cite[Theorem 4.2]{NC-MWS:08} and \cite[Theorem 1]{AJ-NM-MB:04}.
  \oprocend
  \vspace{-2pt}
\end{remark}

Recall from Section \ref{Section: Introduction} that angular differences are well defined for $\theta \in \Delta(\pi)$.  Hence, for $\theta \in \Delta(\pi)$, the vector of phase differences is $H \theta = (\theta_{2}-\theta_{1},\dots) \!\in\! \mbb R^{n(n-1)/2}$, and the function $\mc W$ defined in \eqref{eq: mathcal W(theta)} can be rewritten as the function
$H\theta\mapsto W(H \theta)$ defined by
\begin{equation}
  \mc W(\theta)
  =
  \frac{1}{4} \sum\nolimits_{i=1}^{n}\sum\nolimits_{j=1}^{n} D_{i} D_{j}\, |\theta_{i}-\theta_{j}|^{2}
  =
  \frac{1}{2} (H\theta)^{T} \diag(D_{i}D_{j}) (H\theta)
  =: W(H\theta)
  \label{eq: W(Htheta)}
  \,.
\end{equation}
The derivative of $W(H\theta)$ along trajectories of system \eqref{eq: Non-uniform Kuramoto model - phase
  differences - in vector form} is then given by
\begin{multline}
  	\dot{W}(H\theta)  
	=
  	(H\theta)^{T}\diag(D_{i} D_{j})  H D^{-1} \omega
  	- (H\theta)^{T} \diag(D_{i} D_{j})  H X 
	\\
	- (H\theta)^{T} \diag(D_{i} D_{j})  H D^{-1} H^{T}\diag(P_{ij}\cos(\varphi_{ij})) \sinbf(H\theta)
	\label{eq: dot W(Htheta) non-simplified} \,.
\end{multline}
A component-wise analysis of the last term on the right-hand side of \eqref{eq: dot W(Htheta) non-simplified} yields a ``diagonal''\,\,simplification. 

\begin{lemma}\label{Lemma: Identity for dot W}
Let $P=P^{T} \in \mbb R^{n \times n}$, $\theta \in \Delta(\pi)$, and $\kappa := \sum_{k=1}^{n} D_{k}$. Then it holds\,\,that
\begin{equation}  
    (H\theta)^{T} \diag( D_{i} D_{j})H D^{-1} H^{T}\diag(P_{ij}\cos(\varphi_{ij}))
    \sinbf(H\theta) 
    =    
    \kappa (H\theta)^{T} \diag(P_{ij}\cos(\varphi_{ij})) \sinbf(H\theta) 
    \label{eq: identity for dot W}
    .
\end{equation}
\end{lemma}

\begin{proof}
The left-hand side of equation \eqref{eq: identity for dot W} reads component-wise as 
\begin{equation*}
  	\sum\nolimits_{i} \sum\nolimits_{j} \sum\nolimits_{k} (\theta_{i}-\theta_{j}) (P_{ik}\cos(\varphi_{ik}) D_{j})  \sin(\theta_{i}-\theta_{k}) 
	\;+\;
	\sum\nolimits_{i} \sum\nolimits_{j} \sum\nolimits_{k} (\theta_{i}-\theta_{j})  (P_{jk}\cos(\varphi_{jk}) D_{i}) \sin(\theta_{k}-\theta_{j})
 	\,,
\end{equation*}
where all indices satisfy $i,j,k \in \until n$. An manipulation of the indices in both sums yields
\begin{equation*}  
	\sum\nolimits_{i} \sum\nolimits_{k} \sum\nolimits_{j} (\theta_{i}-\theta_{k}) (P_{ij}\cos(\varphi_{ij}) D_{k})  \sin(\theta_{i}-\theta_{j}) 
	\;+\;
	\sum\nolimits_{k} \sum\nolimits_{j} \sum\nolimits_{i} (\theta_{k}-\theta_{j})  (P_{ij}\cos(\varphi_{ij}) D_{k}) \sin(\theta_{i}-\theta_{j})
	\,.
\end{equation*}
Finally, the two sums can be added and simplify to
$
	  \sum\nolimits_{i} \sum\nolimits_{k} \sum\nolimits_{j} (P_{ij} \cos(\varphi_{ij}) D_{k})  (\theta_{i}-\theta_{j}) \sin(\theta_{i}-\theta_{j})
$,
which equals the right-hand side of equation \eqref{eq: identity for dot W} written in components.
\end{proof}
The following lemma will help us to upper-bound the derivative $\dot W(H \theta)$ by the algebraic connectivity.

\begin{lemma}\label{Lemma: Lemma on bounding Hx}
Consider a connected graph with $n$ nodes induced by $A \!=\! A^{T} \!\in\! \mbb R^{n \times n}$ with incidence matrix $B$ and Laplacian $L(A_{ij})$. For any $x \in \mbb R^{n}$, it holds that
$
(Bx)^{T} \diag(A_{ij}) (Bx) \geq (\lambda_{2}(L(A_{ij}))/n) \norm{Bx}_{2}^{2}
\,.
$
\end{lemma}

\begin{proof}
Let $H$ be the incidence matrix of the complete graph. The Laplacian of the complete graph with uniform weights is then given by $(n \cdot I_{n} -  \fvec 1_{n} \fvec 1_{n}^{T}) = H^{T}H$, and the projection of $x \in \mbb R^{n}$ on the subspace orthogonal to $\fvec 1_{n}$ is $x_{\perp} = (I_{n} - (1/n) \fvec 1_{n} \fvec 1_{n}^{T}) x = (1/n) H^{T}H \, x$. Consider now the inequality
\begin{multline*}
	(Bx)^{T} \diag(A_{ij}) (Bx)
	=
	x^{T} B^{T} \diag(A_{ij}) B x
	=
	x^{T} L(A_{ij}) x
	\\
	\geq
	\lambda_{2}(L(A_{ij})) \norm{x_{\perp}}_{2}^{2}
	=
	\frac{\lambda_{2}(L(A_{ij}))}{n^{2}} \norm{ H^{T} H x}_{2}^{2}
	=
	\frac{\lambda_{2}(L(A_{ij}))}{n^{2}} (H x)^{T} H H^{T} (H x)
	\,.
\end{multline*}
In order to continue, first note that $H H^{T}$ and the complete graph's Laplacian $H^{T} H$ have the same eigenvalues, namely $n$ and $0$. Second, $\mathrm{range}(H)$ and $\mathrm{ker}(H^{T})$ are orthogonal complements. It follows that
$
(H x)^{T} H H^{T} (Hx)
=
n \norm{H x}_{2}^{2}
$.
Finally, note that $\norm{H x}_{2}^{2} \geq \norm{B x}_{2}^{2}$ and the lemma follows.
\end{proof}

Given Lemma \ref{Lemma: Identity for dot W} and Lemma \ref{Lemma: Lemma on bounding Hx} about the time derivative of
$W(H\theta)$, we are now in a position to prove Theorem \ref{Theorem: Synchronization Condition II} via standard Lyapunov and ultimate
boundedness arguments.

{\itshape Proof of Theorem \ref{Theorem: Synchronization Condition II}: }
Assume that $\theta(0) \in \mc S(\rho) := \{\theta \in \Delta(\pi):\, \norm{H\theta}_{2} \leq \rho \} $ for some $\rho \in ]0,\pi[$. In the following, we will show under which conditions and for which values of $\rho$ the set $\mc S(\rho)$ is positively invariant.
For $\theta \in \mc S(\rho)$ and since $\norm{H\theta}_{\infty} \leq \norm{H\theta}_{2}$, it follows that $\theta \in \bar\Delta(\rho)$ and
$1 \geq \sinc(\theta_{i} - \theta_{j}) \geq \sinc(\rho)$. Thus, for $\theta \in \mc S(\rho)$, the inequality
$(\theta_{i} - \theta_{j}) \sin(\theta_{i} - \theta_{j}) \!\geq\! (\theta_{i} - \theta_{j})^{2} \sinc(\rho)$ and Lemma \ref{Lemma: Identity for dot W} 
yield an upper bound on the right-hand side of \eqref{eq: dot W(Htheta) non-simplified}:
\begin{equation*}
  \dot{W}(H\theta)
  \leq
  (H\theta)^{T}\! \diag(D_{i} D_{j})  H D^{-1} \omega
  - (H\theta)^{T}\! \diag(D_{i}D_{j})  H X 
  -  \kappa \sinc(\rho)  (H\theta)^{T}\! \diag(P_{ij}\cos(\varphi_{ij})) (H\theta)
  \,.
\end{equation*}
Note that $\norm{HX}_{2}$ is lower bounded as
\begin{equation*}
	\norm{HX}_{2}
	=
	\sqrt{X^{T} H^{T} H X}
	\geq
	\sqrt{\subscr{\lambda}{max}(H^{T}H)} \norm{X}_{2}
	\geq
	\sqrt{n} \, \Big|\Big|\Big[ \dots, \sum\nolimits_{j} \frac{P_{ij}}{D_{i}}\sin(\varphi_{ij}), \dots \Big]\Big|\Big| _{2}
	=: \tilde X
	\,.
\end{equation*}
This lower bound $\tilde X$ together with Lemma \ref{Lemma: Lemma on bounding Hx} leads to the following upper bound on $\dot W(H\theta)$:
\begin{equation}
  \dot{W}(H\theta)
  \leq
  \norm{H\theta}_{2} \max\nolimits_{i\neq j} \{D_{i} D_{j} \}\big( \norm{H D^{-1} \omega}_{2} + \tilde X \big)
  -  
  (\kappa/n) \sinc(\rho) \lambda_{2}(L(P_{ij} \cos(\varphi_{ij}) ) \norm{H\theta}_{2}^{2}
  \label{eq: dot W(Htheta) simplified}
  \,.
\end{equation}
Note that the right-hand side of \eqref{eq: dot W(Htheta) simplified} is
strictly negative for
\begin{equation*}
  \norm{H\theta}_{2} > 
  \mu_{c} := \frac{ \max\nolimits_{i \neq j} \{D_{i} D_{j} \} (\norm{H D^{-1} \omega}_{2} + \tilde X) }
  { (\kappa/n) \sinc(\rho) \lambda_{2}(L(P_{ij}\cos(\varphi_{ij})))} 
  \,.
\end{equation*}
In the following we apply standard Lyapunov and ISS 
arguments. Pick $\mu \in ]0,\rho[$. If
\begin{equation}
  \mu > \mu_{c} =
  \frac{ \max\nolimits_{i \neq j} \{D_{i} D_{j} \}(\norm{H D^{-1} \omega}_{2} + \tilde X) }
  { (\kappa/n) \sinc(\rho) \lambda_{2}(L(P_{ij}\cos(\varphi_{ij})))} 
  \label{eq: key-assumption - Kuramoto - lossless - used in proof}
  \,,
\end{equation}
then for all $\norm{H\theta}_{2} \in [\mu,\rho]$, 
the right-hand side of \eqref{eq: dot W(Htheta) simplified} is 
upper-bounded by
\begin{align*}
  \dot{W}(H\theta) &\leq - (1-(\mu_{c}/\mu)) \cdot (\kappa/n) \sinc(\rho) \lambda_{2}(L(P_{ij}\cos(\varphi_{ij})))
  \norm{H\theta}_{2}^{2} \,.
\end{align*}
Note that $W(H\theta)$ defined in \eqref{eq: W(Htheta)} can easily be upper and lower bounded by constants
multiplying $\norm{H\theta}_{2}^{2}$:
\begin{equation}
  \min\nolimits_{i \neq j} \{D_{i} D_{j} \} \norm{H\theta}_{2}^{2}
  \leq 
  2 \cdot W(H\theta) 
  \leq  
  \max\nolimits_{i \neq j} \{D_{i} D_{j} \} \norm{H\theta}_{2}^{2}
  \label{eq: bounds on W(theta) and norm(Htheta)}
  \,.
\end{equation}
To guarantee the ultimate boundedness of $H\theta$, two sublevel sets of
$W(H\theta)$ have to be fitted into 
$\{ H\theta: \norm{H\theta}_{2} \in [\mu,\rho] \} $ 
where $\dot W(H\theta)$ is strictly negative. 
This is possible 
if \cite[equation (4.41)]{HKK:02} 
\begin{equation}
  \mu 
  <
  \sqrt{ \min\nolimits_{i \neq j} \{D_{i} D_{j} \} / \max\nolimits_{i \neq j} \{D_{i} D_{j} \} } \cdot \rho
  =
  \alpha \rho
  \label{eq: Squeezing level sets into an interval}
  \,.
\end{equation}
Ultimate boundedness arguments \cite[Theorem 4.18]{HKK:02} imply
that, for every $\norm{H\theta(0)}_{2} \leq \alpha \rho$, there is\,\,$T \geq 0$ such that $\norm{H\theta(t)}_{2}$ is strictly decreasing for $t \in [0,T]$ and $\norm{H\theta(t)}_{2} \leq \mu/\alpha $ for all
$t\geq T$. 
If we choose $\mu = \alpha \gamma$ with $\gamma \in ]0,\pi/2 - \subscr{\varphi}{max}]$, then equation \eqref{eq: Squeezing level sets into an interval} reduces to $\rho > \gamma$
and \eqref{eq: key-assumption - Kuramoto - lossless - used in proof} reduces to the condition
\begin{equation}
 	\lambda_{2}(L(P_{ij}\cos(\varphi_{ij})))
  	>
   	\frac{\max\nolimits_{i \neq j} \{D_{i} D_{j} \}(\norm{H D^{-1} \omega}_{2} + \tilde X)}{\alpha \gamma\, (\kappa/n)\sinc(\rho)} 
	=
	\subscr{\lambda}{critical} \frac{\cos(\subscr{\varphi}{max})}{\gamma \sinc(\rho)}
  	\label{eq: key-assumption - Kuramoto - lossless - used in proof 2}
	\,,
\end{equation}
where $\subscr{\lambda}{critical}$ is as defined in equation \eqref{eq: key-assumption - Kuramoto - 2}.
Now, we perform a final analysis of the bound \eqref{eq: key-assumption - Kuramoto - lossless - used in proof 2}. 
The right-hand side of \eqref{eq: key-assumption - Kuramoto - lossless - used in proof 2} is an increasing function of $\rho$ and decreasing function of $\gamma$ that diverges to $\infty$ as $\rho \uparrow \pi$ or $\gamma \downarrow 0$. Therefore, there exists some $(\rho,\gamma)$ in the convex set 
$\Lambda := \{ (\rho,\gamma) :\, \rho \in ]0,\pi[ \,,\;  \gamma \in ]0,\pi/2 - \subscr{\varphi}{max}] \,,\; \gamma < \rho \}$ 
satisfying equation \eqref{eq: key-assumption - Kuramoto - lossless - used in proof 2} if and only if equation \eqref{eq: key-assumption - Kuramoto - lossless - used in   proof 2} is true at $\rho = \gamma = \pi/2 - \subscr{\varphi}{max}$, where the right-hand side of \eqref{eq: key-assumption - Kuramoto - lossless - used in proof 2} achieves its infimum in $\Lambda$. The latter condition is equivalent to inequality \eqref{eq: key-assumption - Kuramoto - 2}.  Additionally, if these two equivalent statements are true, then there exists an open set of points in $\Lambda$ satisfying \eqref{eq: key-assumption - Kuramoto - lossless - used in proof 2} , which is bounded by the unique curve that satisfies equation \eqref{eq: key-assumption - Kuramoto - lossless - used in proof 2} with the equality sign, namely $f(\rho,\gamma) = 0$, where $f: \Lambda \to \mbb R$,
$f(\rho,\gamma) := \gamma \sinc(\rho)/ \cos(\subscr{\varphi}{max})  -  \subscr{\lambda}{critical}/\lambda_{2}(L(P_{ij} \cos(\varphi_{ij})))$. 
Consequently, for every $(\rho,\gamma) \in \{ (\rho,\gamma) \in \Lambda:\, f(\rho,\gamma) > 0 \}$, it follows for $\norm{H \theta(0)}_{2} \leq \alpha \rho$ that there is $T \geq 0$ such that $\norm{H\theta(t)}_{2} \leq \gamma$ for all $t \geq T$. The supremum value for $\rho$ is obviously given by $\subscr{\rho}{max} \in {]\pi/2-\subscr{\varphi}{max},\pi]}$ solving the equation $f(\subscr{\rho}{max},\pi/2-\subscr{\varphi}{max}) =0$ and the corresponding infimum of $\gamma$ by $\subscr{\gamma}{min} \in  {[0,\pi/2-\subscr{\varphi}{max}[}$ solving the equation $f(\subscr{\gamma}{min},\subscr{\gamma}{min})=0$. 

This proves statement 1) (where we replaced $\subscr{\rho}{max}$ by $\subscr{\gamma}{max}$) and shows that there is $T \geq 0$ such that $\norm{H\theta(t)}_{\infty} \!\leq\! \norm{H\theta(t)}_{2} < \pi/2 - \subscr{\varphi}{max}$ for all $t \geq T$. Statement 2) follows then from Theorem \ref{Theorem: Frequency synchronization}.
\hspace*{\fill}~\QED\par\endtrivlist\unskip

%% file: Sections/PhaseSynchronization.tex
\subsection{Phase Synchronization}\label{Subsection: Phase Synchronization}

For identical natural frequencies and zero phase shifts, the practical stability results in Theorem \ref{Theorem: Synchronization Condition I} and Theorem \ref{Theorem: Synchronization Condition II} imply $\subscr{\gamma}{min} \downarrow 0$, i.e., phase synchronization of the non-uniform Kuramoto oscillators\,\,\eqref{eq: Non-uniform Kuramoto model}.

\begin{theorem}[Phase synchronization]\label{Theorem: Phase Synchronization}
Consider the non-uniform Kuramoto model \eqref{eq: Non-uniform Kuramoto model}, where the graph induced by $P$ has a globally reachable node, $\subscr{\varphi}{max} = 0$, and $\omega_{i}/D_{i} = \bar \omega$ for all $i \in \until n$.

Then for every $\theta(0) \in \bar\Delta(\gamma)$ with $\gamma \in {[0,\pi[}$,
\begin{enumerate}

	\item [1)] the phases $\theta_{i}(t)$ synchronize exponentially to $\theta_{\infty}(t) \in  [\subscr{{\theta}}{min}(0) , \subscr{{\theta}}{max}(0)] + \bar \omega t$, and

	\item [2)] if $P=P^{T}$, then the phases $\theta_{i}(t)$ synchronize exponentially to the weighted mean angle\footnote{This weighted average of angles is geometrically well defined for $\theta(0) \in \Delta(\pi)$.}
	$\theta_{\infty}(t) = \sum_{i} D_{i} \theta_{i}(0)/\sum_{i} D_{i} + \bar \omega t$ at a rate no worse than
  \begin{equation}     
    \subscr{\lambda}{ps} 
    = 
    - \lambda_{2}(L(P_{ij})) \sinc(\gamma) \cos(\angle(D\fvec 1,\fvec 1))^{2} \!/ \subscr{D}{max}
	\label{eq: rate lambda_ps}
    \,.
  \end{equation}

\end{enumerate}
\end{theorem}

The worst-case phase synchronization rate $\subscr{\lambda}{ps}$ can be interpreted similarly as the terms in \eqref{eq: rate lambda_fe}, where $\sinc(\gamma)$ corresponds to the initial phase cohesiveness in $\bar\Delta(\gamma)$. For classic Kuramoto oscillators \eqref{eq: Kuramoto system} statements 1) and 2) can be reduced to the Kuramoto results found in \cite{ZL-BF-MM:07} and Theorem 1 in \cite{AJ-NM-MB:04}.


{\itshape Proof of Theorem \ref{Theorem: Phase Synchronization}: }
First we proof statement 1). Consider again the Lyapunov function $V(\theta(t))$ from the proof of Theorem  \ref{Theorem: Synchronization Condition I}. The Dini derivative in the case $\subscr{\varphi}{max} = 0$ and $\omega_{i}/D_{i} = \bar \omega$ is simply
\begin{equation*}
	D^{+} V (\theta(t))
	=
	- \sum\nolimits_{k=1}^{n} \Big(
	\frac{P_{m k}}{D_{m}} \sin(\theta_{m}(t) - \theta_{k}(t)) 
	+ 
	\frac{P_{\ell k}}{D_{\ell}} \sin(\theta_{k}(t)-\theta_{\ell}(t)) 
	\Big)
	\,.
\end{equation*}
Both sinusoidal terms are positive for $\theta(t) \in \bar\Delta(\gamma)$, $\gamma \in {[0,\pi[}$. Thus, $V(\theta(t)$ is non-increasing, and $\bar\Delta(\gamma)$ is positively invariant. Therefore, the term $a_{ij}(t) = (P_{ij}/D_{i}) \sinc(\theta_{i}(t) - \theta_{j}(t))$  is strictly positive for all $t \geq 0$, and after changing to a rotating frame (via the coordinate transformation $\theta \mapsto \theta - \bar \omega\,t$) the non-uniform Kuramoto model \eqref{eq: Non-uniform Kuramoto model} can be written as the consensus time-varying consensus protocol
\begin{equation}
	\dot \theta_{i}(t) = - \sum\nolimits_{j=1}^{n} a_{ij}(t) (\theta_{i}(t) - \theta_{j}(t))
	\label{eq: consensus for phases}
	\,,
\end{equation}
Statement 1) follows directly along the lines of the proof of statement 1) in Theorem \ref{Theorem: Frequency synchronization}. In the case of symmetric coupling $P=P^{T}$, the phase dynamics \eqref{eq: consensus for phases} can be reformulated as a {\it symmetric} time-varying consensus protocol with strictly positive weights $w_{ij}(t) = P_{ij} \sinc(\theta_{i}(t) - \theta_{j}(t))$ and multiple rates $D_{i}$\,\,as
\begin{equation}
	\dt D \theta
	=
	- L(w_{ij}(t)) \, \theta
	\label{eq: LPV concensus system for dot theta - trivial phase shifts}
	\,,
\end{equation}
Statement 2) now follows directly along the lines of the proof of statement 2) in Theorem \ref{Theorem: Frequency synchronization}.
\footnote{The proof of Theorem \ref{Theorem: Synchronization Condition II} can be extended for $HD^{-1} \omega = \fvec 0$ and $X = \fvec 0$ to show statement 2) of Theorem \ref{Theorem: Phase Synchronization} with a slightly different worst-case synchronization frequency than \eqref{eq: rate lambda_ps}.}
\hspace*{\fill}~\QED\par\endtrivlist\unskip

The main result Theorem \ref{Theorem: Main Synchronization Result} can be proved now as a corollary of Theorem \ref{Theorem: Synchronization Condition I} and Theorem \ref{Theorem: singular perturbations}.

{\itshape Proof of Theorem \ref{Theorem: Main Synchronization Result}: }
The assumptions of Theorem \ref{Theorem: Main Synchronization Result} correspond exactly to the assumptions of Theorem \ref{Theorem: Synchronization Condition I} and statements 1) and 2) of Theorem \ref{Theorem: Main Synchronization Result} follow trivially from Theorem \ref{Theorem: Synchronization Condition I}.

Since the non-uniform Kuramoto model synchronizes exponentially and achieves phase cohesiveness in $\bar\Delta(\subscr{\gamma}{min}) \subsetneq \Delta(\pi/2 - \subscr{\varphi}{max})$, it follows from Lemma \ref{Lemma: Properties of grounded Kuramoto model} that the grounded non-uniform Kuramoto dynamics \eqref{eq: grounded Kuramoto model} converge exponentially to a stable fixed point $\delta_{\infty}$. Moreover, $\delta(0)=\groundedphases(\theta(0))$ is bounded and thus necessarily in a compact subset of the region of attraction of the fixed point $\delta_{\infty}$. Thus, the assumptions of Theorem \ref{Theorem: singular perturbations} are satisfied. Statements 3) and 4) of Theorem \ref{Theorem: Main Synchronization Result} follow from Theorem \ref{Theorem: singular perturbations}, where we made the following changes: the approximation errors \eqref{eq: singular perturbation error 1}-\eqref{eq: singular perturbation error 2} are expressed as the approximation\,\,errors \eqref{eq: approx-errors} in $\theta$-coordinates, we stated only the case $\epsilon < \epsilon^{*}$ and $t \geq t_{b}>0$, we reformulated $h(\bar \delta(t)) = D^{-1} Q(\bar\theta(t))$, and weakened the dependence of $\epsilon$ on $\Omega_{\delta}$ to a dependence on $\theta(0)$.
\hspace*{\fill}~\QED\par\endtrivlist\unskip

%% file: Sections/Simulations.tex
Figure \ref{Fig: Simulation} shows a simulation of the power network model \eqref{eq: Classical model} with $n=10$ generators and the corresponding non-uniform Kuramoto model \eqref{eq: Non-uniform Kuramoto model}, where all initial angles $\theta(0)$ are clustered with exception of the first one (red curves) and the initial frequencies are chosen as $\dot \theta(0) \in \mathrm{uni}(-0.1,0.1)$\,rad/s, i.e., randomly from a uniform distribution over $[-0.1,0.1]$. Additionally, at two-third of the simulation interval a transient high frequency disturbance is introduced at $\omega_{n-1}$ (yellow curve). For illustration, relative angular coordinates are defined as $\delta_{i}(t) = \theta_{i}(t) - \theta_{n}(t)$, $i \in \until{n-1}$. The system parameters\,\,satisfy $\omega_{i} \in \mathrm{uni}(0,10)$, $P_{ij}  \in \mathrm{uni}(0.7,1.2)$, and $\tan(\varphi_{ij})  \in \mathrm{uni}(0,0.25)$ matching data found in \cite{MAP:89,PMA-AAF:77,PK:94}. 

For the simulation in Figure \ref{Fig: underdamped case}, we chose $M_{i}  \in \mathrm{uni}(2,12) s/(2 \pi f_{0})$ and $D_{i}  \in \mathrm{uni}(20,30) s/(2 \pi f_{0})$ resulting in the rather large perturbation parameter $\epsilon = 0.58$\,s. The synchronization conditions of Theorem \ref{Theorem: Main Synchronization Result} are satisfied, and the angles $\bar \delta(t)$ of the non-uniform Kuramoto model synchronize very fast from the non-synchronized initial conditions (within $0.05$\,s), and the disturbance around $t=2$s does not severely affect the synchronization dynamics. The same findings hold for the quasi-steady state $h(\bar \delta)$ depicting the frequencies of the non-uniform Kuramoto model, where the the disturbance at angle $n-1$ (yellow curve) acts directly without being integrated. Since $\epsilon$ is large the power network trajectories $(\delta(t),\dot \theta(t))$ show the expected underdamped behavior and synchronize with second-order dynamics. As expected, the disturbance at $t=2$\,s does not affect the second order power network $\delta$-dynamics as much as the first-order non-uniform Kuramoto $\bar\delta$-dynamics. Nevertheless, after the initial and mid-simulation transients the singular perturbation errors $\delta(t) - \bar \delta(t)$ and $\theta(t) - h(\bar \delta(t))$ quickly become small and ultimately converge.

Figure \ref{Fig: overdamped case} shows the exact same simulation as in Figure \ref{Fig: underdamped case}, except that the simulation time is halved, the inertia are $M_{i}  \in \mathrm{uni}(2,6)s / (2 \pi f_{0})$, and the damping is chosen uniformly as $D_{i} \!=\! 30 s/(2 \pi f_{0})$, which gives the small perturbation parameter $\epsilon \!=\! 0.18$\,s. The resulting power network dynamics $(\delta(t),\dot \theta(t))$ are strongly damped (note the different time scales), and the non-uniform Kuramoto dynamics $\bar \delta(t)$ and the quasi-steady state $h(\bar \delta(t))$ have slower time constants. As expected, the singular perturbation errors remain smaller during transients and converge faster than in the weakly damped case in Figure\,\ref{Fig: underdamped case}.

\begin{figure}[t]
  \centering
  \subfigure[Weakly damped simulation with $\epsilon = 0.58$\,s]
  {
  \includegraphics[scale=0.49]{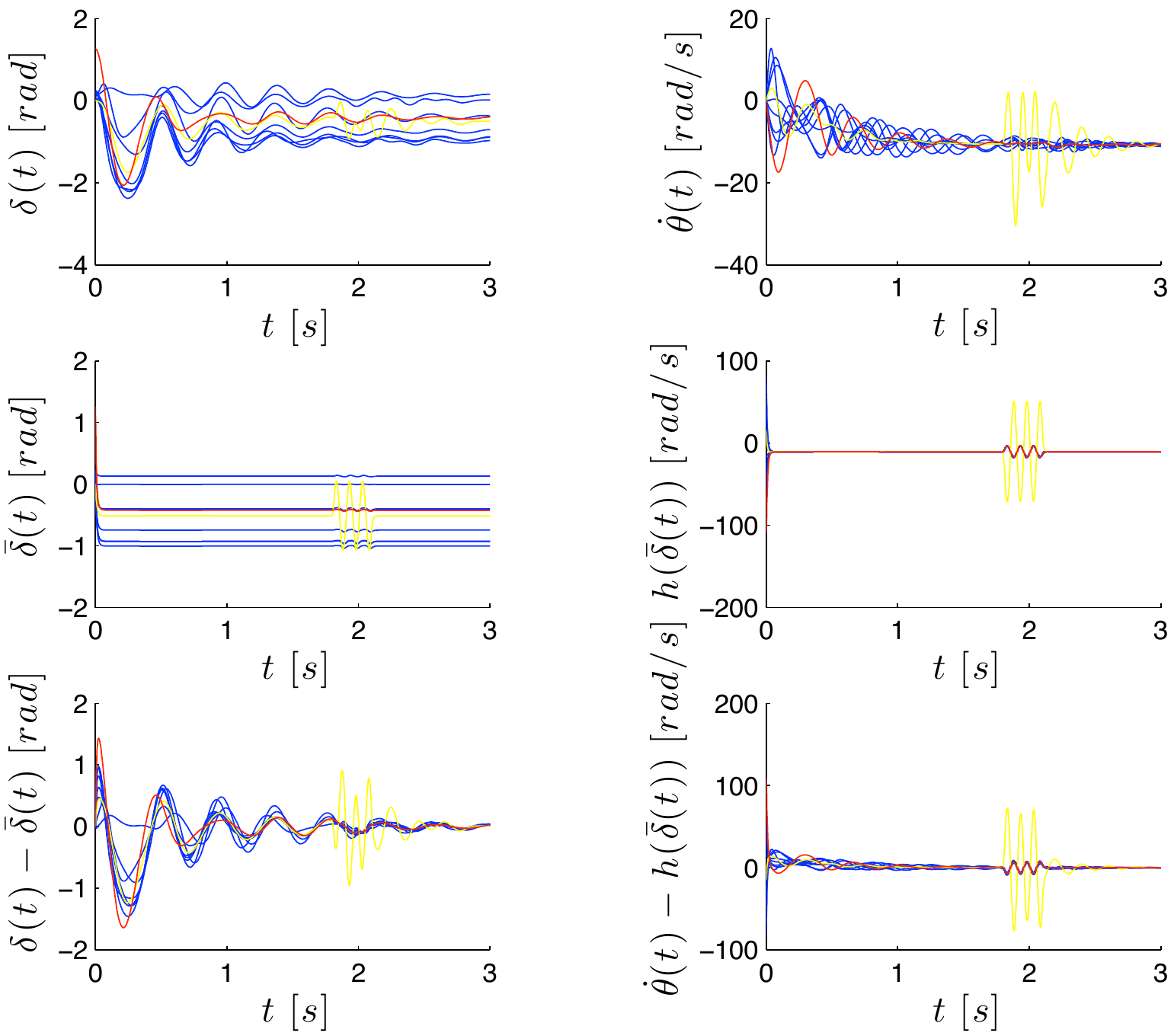}
    \label{Fig: underdamped case}
    }
   \hspace{0.35cm}
   \subfigure[Strongly damped simulation with $\epsilon = 0.18$\,s]
  {
  \includegraphics[scale=0.49]{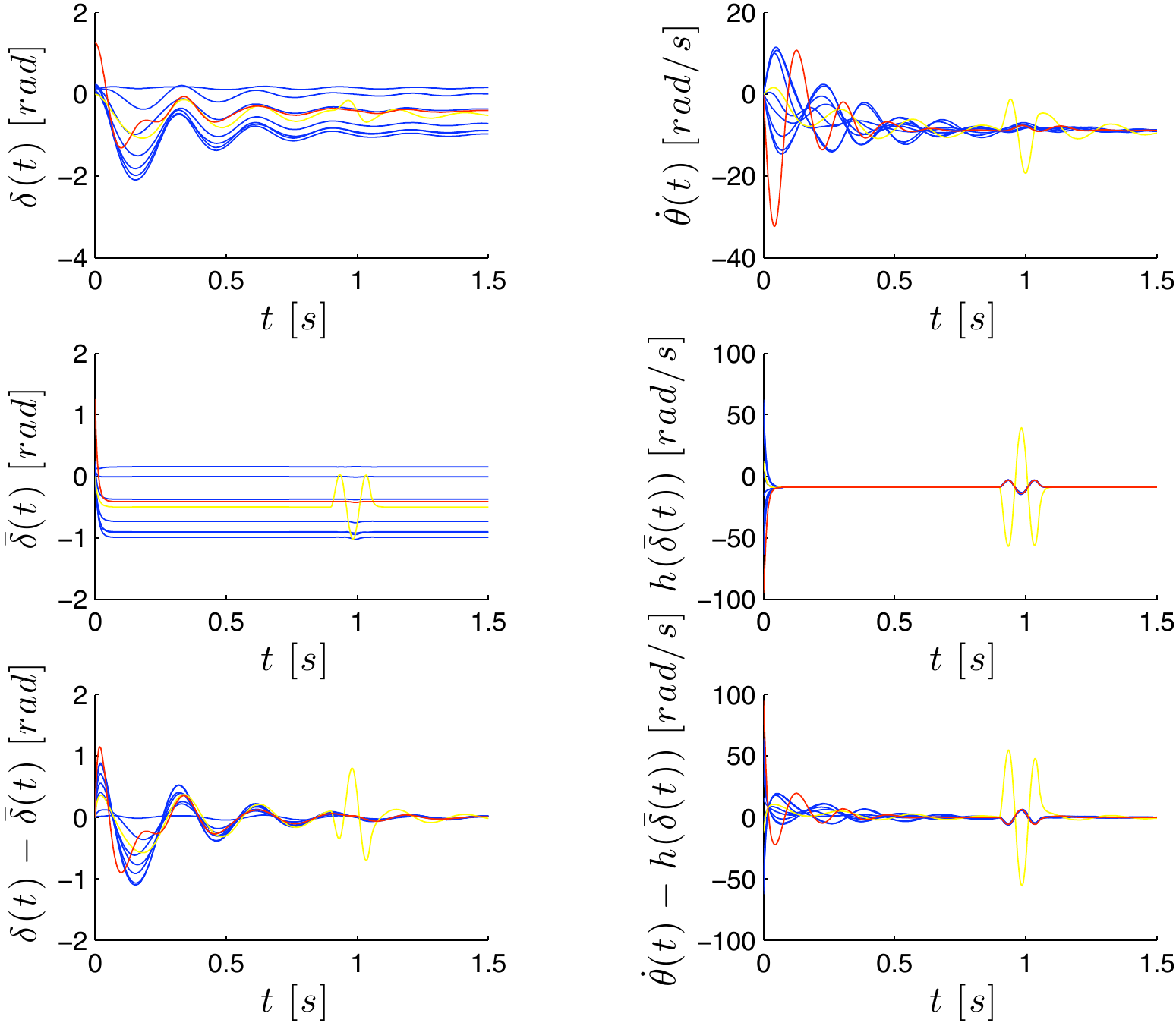}
    \label{Fig: overdamped case}
    }
    \caption{Simulation of the power network model \eqref{eq: Classical model} and the non-uniform Kuramoto model \eqref{eq: Non-uniform Kuramoto model}}
     \label{Fig: Simulation}
     \vspace{-10pt}
\end{figure}

%% file: Sections/Conclusions.tex
This paper studied the synchronization and transient stability problem for
a power network. A novel approach leads to purely algebraic
conditions, under which a network-reduced power
system model is transiently stable depending on network parameters and
initial phase differences. Our technical approach is based on the
assumption that each generator is highly overdamped due to local excitation
control. The resulting singular perturbation analysis leads to the
successful marriage of transient stability in power networks, Kuramoto
oscillators, and consensus protocols. As a result, the transient stability analysis of a power network
model reduces to the synchronization analysis of non-uniform Kuramoto
oscillators. The study of generalized coupled oscillator models is an
interesting mathematical problem in its own right and was tackled by
combining and extending different techniques from all three mentioned
areas.

The presented approach to synchronization in power networks offers easily checkable conditions and an entirely new perspective on the transient stability problem. The authors are aware that the derived conditions are not yet competitive with the sophisticated numerical algorithms developed by the power systems community. To render our results applicable to real power systems, tighter synchronization conditions have to be developed, the region of attraction has to be characterized more accurately, and more realistic power network models have to be considered. The authors' ongoing work addresses the last point and extends the presented analysis to structure-preserving power network models. 

Finally, the revealed relationship between power networks, Kuramoto oscillators, and consensus algorithms gives rise to various exciting research directions at the interface between these areas.

%% file: Sections/Appendix.tex
In this appendix we briefly comment on alternative bounding methods in the proof of Theorem \ref{Theorem: Synchronization Condition I} and how they affect the triplet of the synchronization condition \eqref{eq: key-assumption}, the
estimate for the region of attraction $\Delta(\subscr{\gamma}{max})$, and
the ultimate phase cohesive set $\bar\Delta(\subscr{\gamma}{min})$. We will state only the essential parts of the theorem statements and the corresponding proofs.

\subsection{Pairwise Bounding}

The proof of Theorem \ref{Theorem: Synchronization Condition I} can be continued from equation \eqref{eq: D+ with cosine terms} by bounding the right-hand side of equation \eqref{eq: D+ with cosine terms} for each single pair $\{m,\ell\}$ rather than for all $m,\ell \in \until n$. Such a pairwise bounding results in $n(n-1)/2$ pairwise synchronization conditions and the worst multiplicative gap over all conditions determines the estimates for the region of attraction $\Delta(\subscr{\gamma}{max})$ and the ultimate phase cohesive set $\bar\Delta(\subscr{\gamma}{min})$. In short, tighter bounds are traded off for complexity. The resulting theorem statement is as follows.

\begin{theorem}
\label{Theorem: Sync Condition I - Appendix - 2}
{\bf(Synchronization condition I)}
Consider the non-uniform Kuramoto-model \eqref{eq: Non-uniform Kuramoto
  model}, where the graph induced by $P=P^{T}$ is complete. Assume that the minimal lossless coupling of any oscillator pair $\{m,\ell\}$ to the network is
larger than a critical value, i.e., for every $m,\ell \in \until n$, $m \neq \ell$,
\begin{multline}
	\Gamma_{m\ell}
	:=
	\sum\limits_{k=1}^{n} \min_{i \in \{m,\ell\} \setminus\{k\} } \left\{ \frac{P_{i k}}{D_{i}} \cos(\varphi_{i k}) \right\}
	>\\
	\supscr{\Gamma_{m\ell}}{critical}
	:=
	\frac{1}{\cos(\subscr{\varphi}{max})} 
	\cdot \left(
	\left| \frac{\omega_{m}}{D_{m}} - \frac{\omega_{\ell}}{D_{\ell}} \right|
	+
	\sum\limits_{k=1}^{n} \left( \frac{P_{m k}}{D_{m}} \sin(\varphi_{mk}) + \frac{P_{\ell k}}{D_{\ell }} \sin(\varphi_{\ell k}) \right)
	\right)
  	\label{eq: key-assumption - Kuramoto - Appendix - 2}
  	\,.
\end{multline}
Accordingly, define $\subscr{\gamma}{min} \in [0,\pi/2 - \subscr{\varphi}{max}[$ and $\subscr{\gamma}{max} \in {]\pi/2,\pi]}$ as unique solutions to 
the equations $\sin(\subscr{\gamma}{min}) = \sin(\subscr{\gamma}{max}) = \cos(\subscr{\varphi}{max}) \max_{m,\ell} \{\supscr{\Gamma_{m\ell}}{critical} / \Gamma_{m\ell}\}$. Then \dots
\end{theorem}

{\itshape Proof of Theorem  \ref{Theorem: Sync Condition I - Appendix - 2}: }
\dots [see proof of Theorem \ref{Theorem: Synchronization Condition I}]  \dots\\
In summary, $D^{+}V(\theta(t))$ in \eqref{eq: D+ with cosine terms} can be upper bounded by the simple expression
\begin{equation*}
	D^{+} V (\theta(t)) 
	\leq\, 
	\frac{\omega_{m}}{D_{m}} - \frac{\omega_{\ell}}{D_{\ell}} -
  	\sum\nolimits_{k=1}^{n}  
	\min_{i \in \{m,\ell\} \setminus \{k\}} \left\{ a_{ik}  \right\} \sin(\gamma)
	+
	\sum\nolimits_{k}
	b_{\ell k}
	+
	\sum\nolimits_{k}
	b_{m k}
	\,.
\end{equation*}
It follows that $V(\theta(t))$ is non-increasing for all $\theta(t) \in \bar\Delta(\gamma)$ and for all pairs $\{m,\ell\}$ if
\begin{equation}
	\Gamma_{m \ell} \sin(\gamma)
	\geq
	\cos(\subscr{\varphi}{max}) \supscr{\Gamma_{m\ell}}{critical}
	\label{eq: D+ inequality - Appendix - 2}
	\,,
\end{equation}
where $\Gamma_{m \ell}$ and $\supscr{\Gamma_{m \ell}}{critical}$ are defined in \eqref{eq: key-assumption - Kuramoto - Appendix - 2}.
The left-hand side of \eqref{eq: D+ inequality - Appendix - 2} is a strictly concave function of $\gamma \in [0,\pi]$. Thus, there exists an open set of arc lengths $\gamma$ including $\gamma^{*} = \pi/2-\subscr{\varphi}{max}$ satisfying equation \eqref{eq: D+ inequality - Appendix - 2} if and only if equation \eqref{eq: D+ inequality - Appendix - 2} is true at $\gamma^{*}=\pi/2-\subscr{\varphi}{max}$ with the strict inequality sign, which corresponds to condition \eqref{eq: key-assumption - Kuramoto - Appendix - 2} in the statement of Theorem \ref{Theorem: Sync Condition I - Appendix - 2}. Additionally, if these two equivalent statements are true, then $V(\theta(t))$ is non-increasing in $\bar\Delta(\gamma)$ for all $\gamma \in [\subscr{\gamma}{min},\subscr{\gamma}{max}]$, where $\subscr{\gamma}{min} \in {[0,\pi/2 - \subscr{\varphi}{max}[}$ and $\subscr{\gamma}{max} \in {]\pi/2,\pi]}$ satisify inequality \eqref{eq: D+ inequality - Appendix - 2} for all pairs $\{m,\ell\}$ \dots
\hspace*{\fill}~\QED\par\endtrivlist\unskip

\subsection{Pairwise Concavity-Based Bounding}

The bounding in the proof of Theorem \ref{Theorem: Synchronization Condition I} can be further tightened by cocavity-based arguments. This bounding results in $n(n-1)/2$ pairwise synchronization conditions and $n(n-1)$ equations determining the estimates for the region of attraction $\Delta(\subscr{\gamma}{max})$ and
the ultimate phase cohesive set $\bar\Delta(\subscr{\gamma}{min})$. Again, tighter bounds are traded off for increasing complexity. The resulting theorem statement is as follows.

\begin{theorem}
\label{Theorem: Sync Condition I - Appendix - 1}
{\bf(Synchronization condition I)}
Consider the non-uniform Kuramoto-model \eqref{eq: Non-uniform Kuramoto
  model}, where the graph induced by $P=P^{T}$ is complete. Assume that the minimal lossless coupling of any oscillator pair $\{m,\ell\}$ to the network is
larger than a critical value, i.e., for every $m,\ell \in \until n$, $m \neq \ell$,
\begin{equation}
	\sum\limits_{k=1}^{n} \min_{i \in \{m,\ell\}\setminus\{k\} } \!\left\{ \frac{P_{i k}}{D_{i}} \cos(\varphi_{i k} + \subscr{\varphi}{max}) \right\}
	>
	\supscr{\Gamma_{m \ell}}{critical}
	:=
	\left| \frac{\omega_{m}}{D_{m}} - \frac{\omega_{\ell}}{D_{\ell}} \right|
	+
	\max_{i \in \{m,\ell \}} \!\left\{ \sum\limits_{k=1}^{n} \frac{P_{i k}}{D_{i}} \sin(\varphi_{ik}) \right\}
  	\label{eq: key-assumption - Kuramoto - Appendix - 1}
  	\,.
\end{equation}
Accordingly, for every pair $\{m,l\}$ define $\subscr{\gamma}{min}^{m\ell} \in {[0, \pi/2-\subscr{\varphi}{max}[}$ and $\subscr{\gamma}{max}^{m \ell} \in {]\pi/2, \pi]}$ as   unique\,solutions\,to
\begin{equation}
	\sum\limits_{k=1}^{n} \min_{i \in \{m,\ell\}\setminus\{k\} } \!\left\{ \frac{P_{i k}}{D_{i}} \sin(\subscr{\gamma}{min}^{m \ell} - \varphi_{i k}) \right\}
	=
	\sum\limits_{k=1}^{n} \min_{i \in \{m,\ell\}\setminus\{k\} } \!\left\{ \frac{P_{i k}}{D_{i}} \sin(\subscr{\gamma}{max}^{m \ell} + \varphi_{i k}) \right\}
	=
	\supscr{\Gamma_{m \ell}}{critical}
	\label{eq: D+ inequality ** - Appendix - 1}
	\,,
\end{equation}
and let $\subscr{\gamma}{min} := \max_{m,\ell} \{\subscr{\gamma}{min}^{m\ell}\}$ and $\subscr{\gamma}{max} := \min_{m,\ell} \{\subscr{\gamma}{max}^{m \ell} \}$. Then \dots
\end{theorem}

{\itshape Proof of Theorem  \ref{Theorem: Sync Condition I - Appendix - 1}: }
\dots [see proof of Theorem \ref{Theorem: Synchronization Condition I}]  \dots\\
Written out in components $D^{+} V (\theta(t))$ (in the non-expanded form \eqref{eq: Non-uniform Kuramoto model}) takes the form
\begin{equation}
	D^{+} V (\theta(t))
	=
	\frac{\omega_{m}}{D_{m}} - \frac{\omega_{\ell}}{D_{\ell}} - 
	\sum\limits_{k=1}^{n} 
	\frac{P_{m k}}{D_{m}} \sin(\theta_{m}(t) - \theta_{k}(t) + \varphi_{mk}) 
	+ 
	\frac{P_{\ell k}}{D_{\ell}} \sin(\theta_{k}(t)-\theta_{\ell}(t) - \varphi_{\ell k}) 
	\label{eq: D+ in components - Appendix - 1}
	\,.
\end{equation}
In the following we abbreviate the summand on the right-hand side of \eqref{eq: D+ in components - Appendix - 1} as
$
f_{k}(\theta_{k}(t)) 
:=
\frac{P_{m k}}{D_{m}} \sin(\theta_{m}(t) - \theta_{k}(t) + \varphi_{mk}) 
	+ 
\frac{P_{\ell k}}{D_{\ell}} \sin(\theta_{k}(t)-\theta_{\ell}(t) - \varphi_{\ell k}) 
$ 
and aim at a least conservative bounding of $f_{k}(\theta_{k}(t))$. There are different ways to continue from here -- the following approach is based on concavity. 

Since $f_{k}(\theta_{k})$ is the sum of two shifted {\it concave} sine functions of $\theta_{k} \in [\theta_{\ell},\theta_{m}]$ (with same period and shifted by strictly less than $\pi/2$), $f_{k}(\theta_{k})$ is again concave sine function of $\theta_{k} \in [\theta_{\ell},\theta_{m}]$ and necessarily achieves its minimum at the boundary $\theta_{k} \in \{\theta_{\ell},\theta_{m} \}$. 
If $\argmin_{\theta_{k} \in [\theta_{m},\theta_{\ell}]} f_{k}(\theta_{k}) = \theta_{\ell}$, then
\begin{equation*}
	f_{k}(\theta_{k}(t))
	\geq
	f_{k}^{1}(\gamma)  
	:=
	\frac{P_{m k}}{D_{m}} \sin(\gamma + \varphi_{m k}) 
	-
	\frac{P_{\ell k}}{D_{\ell}} \sin(\varphi_{\ell k})  	
	\,,
\end{equation*}
and otherwise for $\argmin_{\theta_{k} \in [\theta_{m},\theta_{\ell}]} f_{k}(\theta_{k}) = \theta_{m}$
\begin{equation*}
	f_{k}(\theta_{k}(t))
	\geq
	f_{k}^{2}(\gamma)  
	:=
	\frac{P_{\ell k}}{D_{\ell}} \sin(\gamma - \varphi_{\ell k}) 
	+ 
	\frac{P_{m k}}{D_{m}} \sin(\varphi_{mk})   	
	\,,
\end{equation*}
where $f_{k}^{1}$ and $f_{k}^{2}$ are functions from $[0,\pi]$ to $\mbb R$. In the following let $f_{k}^{3}:\, [0,\pi] \to \mbb R$ be defined by
\begin{equation*}
	f_{k}^{3}(\gamma)
	:=
	\min \left\{ \frac{P_{mk}}{D_{m}} \sin(\gamma + \varphi_{m k}) , \frac{P_{\ell k}}{D_{\ell}} \sin(\gamma - \varphi_{\ell k}) \right\}
	- 
	\frac{P_{\ell k}}{D_{\ell}} \sin(\varphi_{\ell k})   	
	\,.
\end{equation*}
Since $f_{k}(\theta(t)) \geq f_{k}^{3}(\gamma)$ for all $\theta_{k}(t) \in [\theta_{m}(t),\theta_{\ell}(t)]$, the derivative \eqref{eq: D+ in components - Appendix - 1} is upper-bounded by
\begin{equation*}
	D^{+} V (\theta(t))
	=
	 \frac{\omega_{m}}{D_{m}} - \frac{\omega_{\ell}}{D_{\ell}} 
	-
	\sum\nolimits_{k} f_{k}^{3}(\gamma)
	\,.
\end{equation*} 
It follows that $V(\theta(t))$ is non-increasing for all $\theta(t) \in \bar\Delta(\gamma)$ and for all pairs $\{m,\ell\}$ if for all $\{m,\ell\}$
\begin{equation}
	\sum\limits_{k=1}^{n} \min_{i \in \{m,\ell\} \setminus\{k\} } \!\left\{ \frac{P_{i k}}{D_{i}} \sin(\gamma + \varphi_{i k}) , \frac{P_{i k}}{D_{i}} \sin(\gamma - \varphi_{i k}) \right\}
	\geq
	\left| \frac{\omega_{m}}{D_{m}} - \frac{\omega_{\ell}}{D_{\ell}} \right|
	+
	\max_{i \in \{m,\ell \}} \!\left\{ \sum\limits_{k=1}^{n} \frac{P_{i k}}{D_{i}} \sin(\varphi_{ik}) \right\}
	\label{eq: D+ inequality - Appendix - 1}
	\,.
\end{equation}
Note that the minimizing summand on the left-hand side of \eqref{eq: D+ inequality - Appendix - 1} is $\min_{i \in \{m,\ell\} \setminus\{k\}} \!\left\{ P_{i k} \sin(\gamma - \varphi_{i k}) /D_{i} \right\}$ for $\gamma < \pi/2$, $\min_{i \in \{m,\ell\}\setminus\{k\} } \!\left\{ P_{i k} \sin(\gamma + \varphi_{i k}) /D_{i} \right\}$ for $\gamma > \pi/2$, and it achieves its maximum value $\min_{i \in \{m,\ell\}\setminus\{k\} } \!\left\{ P_{i k} \cos(\varphi_{i k}) /D_{i} \right\}$ for $\gamma = \pi/2$. In particular, there exists an open set of arc lengths $\gamma$ for including $\gamma^{*} = \pi/2 - \subscr{\varphi}{max}$ for which $V(\theta(t))$ is non-increasing in $\bar\Delta(\gamma)$ if and only if inequality \eqref{eq: D+ inequality - Appendix - 1} is strictly satisfied for $\gamma^{*} = \pi/2 - \subscr{\varphi}{max}$. In this case, define $\subscr{\gamma}{min}^{m\ell} \in {[0,\pi/2 - \subscr{\varphi}{max}[}$ and $\subscr{\gamma}{max}^{m\ell} \in {]\pi/2,\pi]}$ as the two unique solutions to equation \eqref{eq: D+ inequality - Appendix - 1} with equality sign, which is equivalent to equation \eqref{eq: D+ inequality ** - Appendix - 1}. Then $V(\theta(t))$ is non-increasing in $\bar\Delta(\gamma)$ for all $\gamma \in [\subscr{\gamma}{min}^{m\ell},\subscr{\gamma}{max}^{m\ell}]$. Finally define $\subscr{\gamma}{min}$ and $\subscr{\gamma}{max}$ as the maximum and minimum values of $\subscr{\gamma}{min}^{m\ell}$ and $\subscr{\gamma}{max}^{m\ell}$ over all pairs $\{m,\ell\}$ \dots
\hspace*{\fill}~\QED\par\endtrivlist\unskip

\subsection{Adding and Subtracting the Lossless Coupling}

The last possible bounding we explore is restricted to the set of initial conditions in $\Delta(\pi/2 - \subscr{\varphi}{max})$ and results in a simple, scalar, and intuitive but very conservative synchronization condition. Instead of bounding the right-hand side of $D^{+} V (\theta(t))$ directly, we add and subtract the lossless coupling in the proof of Theorem \ref{Theorem: Synchronization Condition I}. The resulting theorem statement is as follows.

\begin{theorem}
\label{Theorem: Sync Condition 1 - Appendix - 3}
{\bf(Synchronization condition I)}
Consider the non-uniform Kuramoto-model \eqref{eq: Non-uniform Kuramoto
  model}, where the graph induced by $P=P^{T}$ is complete. Assume that the minimal coupling is
larger than a critical value, i.e., for every $i,j \in \until n$
\begin{equation}
  \label{eq: key-assumption - Kuramoto - Appendix - 3}
  \subscr{P}{min} > \subscr{P}{critical} := 
  \frac{\subscr{D}{max}}{n\cos(\subscr{\varphi}{max})}
  \left(
    \max_{\{i,j\}} \Bigl| \frac{\omega_{i}}{D_{i}}  -
    \frac{\omega_{j}}{D_{j}}  \Bigr|
    +
    \max_i \sum\nolimits_{j=1}^{n} \frac{P_{ij}}{D_{i}}
    \sin(\varphi_{ij})
  \right)
  \,.
\end{equation}
Accordingly, define $\subscr{\gamma}{min} = \arcsin(\cos(\subscr{\varphi}{max})
{\subscr{P}{critical}}/{\subscr{P}{min}})$ taking value in
${[0,\pi/2-\subscr{\varphi}{max}[}$ and $\subscr{\gamma}{max} = \pi/2-\subscr{\varphi}{max}$. Then \dots 
\end{theorem}

{\itshape Proof of Theorem  \ref{Theorem: Sync Condition 1 - Appendix - 3}: }
\dots see proof of Theorem \ref{Theorem: Synchronization Condition I}  \dots\\
Written out in components (in the non-expanded form \eqref{eq: Non-uniform Kuramoto model}) $D^{+} V (\theta(t))$ takes the form
\begin{equation}
	D^{+} V (\theta(t))
	=
	\frac{\omega_{m}}{D_{m}} - \frac{\omega_{\ell}}{D_{\ell}} - 
	\sum\limits_{k=1}^{n} \left( 
	\frac{P_{m k}}{D_{m}} \sin(\theta_{m}(t) - \theta_{k}(t) + \varphi_{mk}) 
	- 
	\frac{P_{\ell k}}{D_{\ell}} \sin(\theta_{\ell}(t)-\theta_{k}(t) + \varphi_{\ell k}) 
	\right)
	\label{eq: D+ - Appendix - 3}
	\,.
\end{equation}
Adding and subtracting the coupling with zero phase shifts yields
\begin{align*}
	D^{+} V (\theta(t))
	=&\;
	\frac{\omega_{m}}{D_{m}} - \frac{\omega_{\ell}}{D_{\ell}} - 
	\sum\nolimits_{k=1}^{n} \Big(
	\frac{P_{m k}}{D_{m}} \sin(\theta_{m}(t) - \theta_{k}(t)) 
	+ 
	\frac{P_{\ell k}}{D_{\ell}} \sin(\theta_{k}(t)-\theta_{\ell}(t)) 
	\Big)
	\\
	&-
	\sum\nolimits_{k=1}^{n} \frac{P_{m k}}{D_{m}}  \Big(
	\sin(\theta_{m}(t) - \theta_{k}(t) + \varphi_{mk}) - \sin(\theta_{m}(t) - \theta_{k}(t)) 
	\Big)
	\\
	&+
	\sum\nolimits_{k=1}^{n} \frac{P_{\ell k}}{D_{\ell}} \Big(
	\sin(\theta_{\ell}(t)-\theta_{k}(t)) + \varphi_{\ell k}) + \sin(\theta_{k}(t)-\theta_{\ell}(t)) 
	\Big)
	\,.
\end{align*}
Since both sinusoidal terms in the first sum are strictly positive, they can be lower-bounded as
\begin{align*} 
	\frac{P_{m k}}{D_{m}} \sin(\theta_{m}(t) - \theta_{k}(t)) 
	+ 
	\frac{P_{\ell k}}{D_{\ell}} \sin(\theta_{k}(t)-\theta_{\ell}(t)) 
	\geq
	\frac{\subscr{P}{min}}{\subscr{D}{max}} \Big( 
	\sin(\theta_{m}(t) - \theta_{k}(t)) 
	+ 
	\sin(\theta_{k}(t)-\theta_{\ell}(t)) 
	\Big)
	\,.
\end{align*}
In the following we apply classic trigonometric arguments from the
Kuramoto literature \cite{NC-MWS:08,FDS-DA:07,GSS-UM-FA:09}.  The identity
$\sin(x) + \sin(y) = 2 \sin(\frac{x+y}{2}) \cos(\frac{x-y}{2})$ leads to
the further simplifications
\begin{align*}
	\sin(\theta_{m}(t) - \theta_{k}(t)) + \sin(\theta_{k}(t)-\theta_{\ell}(t))
	&=
	2\, \sin\!\left(\frac{\theta_{m}(t) - \theta_{\ell}(t)}{2}\right) 
	\cos\!\left(\frac{\theta_{m}(t) + \theta_{\ell}(t)}{2} - \theta_{k}(t) \right), 
	\\
	\sin(\theta_{m}(t) - \theta_{k}(t) + \varphi_{mk}) + \sin(\theta_{k}(t)-\theta_{m}(t))
	&=
	2\, \sin\!\left(\frac{\varphi_{mk}}{2}\right) 
	\cos\!\left(\theta_{m}(t) - \theta_{k}(t) + \frac{\varphi_{mk}}{2}\right), 
	\\
	\sin(\theta_{\ell}(t) - \theta_{k}(t) + \varphi_{\ell k}) + \sin(\theta_{k}(t)-\theta_{\ell}(t))
	&=
	2\, \sin\!\left(\frac{\varphi_{\ell k}}{2}\right) 
	\cos\!\left(\theta_{\ell}(t) - \theta_{k}(t) + \frac{\varphi_{\ell k}}{2}\right) 
	.
\end{align*}
Note that the right-hand side of \eqref{eq: D+ - Appendix - 3} is a convex function of $\theta_{k} \in [\theta_{\ell},\theta_{m}]$ (see Proof of Theorem  \ref{Theorem: Sync Condition I - Appendix - 1}) and accordingly achieves its maximum at the boundary for $\theta_{k} \in \{\theta_{\ell},\theta_{m}\}$
Therefore, $D^{+}V(\theta(t))$ is upper bounded by
\begin{multline*}
  D^{+} V (\theta(t)) \leq\, \max_{\{i,j\}} \Bigl| \frac{\omega_{i}}{D_{i}}
  - \frac{\omega_{j}}{D_{j}} \Bigr| -
  \frac{\subscr{P}{min}}{\subscr{D}{max}} \sum\nolimits_{k=1}^{n} 2 \sin
  \Bigl( \frac{\gamma}{2} \Bigr) \cos
  \Bigl( \frac{\gamma}{2} \Bigr)
  \\
  - \sum\nolimits_{k=1}^{n} \frac{P_{m k}}{D_{m}} 2\,
  \sin\Bigl(\frac{\varphi_{mk}}{2}\Bigr) \cos\Bigl(\gamma + \frac{\varphi_{mk}}{2}\Bigr) 
  + 
  \sum\nolimits_{k=1}^{n}
  \frac{P_{\ell k}}{D_{\ell}} 2\, \sin\Bigl(\frac{\varphi_{\ell
      k}}{2}\Bigr) \cos\Bigl(\frac{\varphi_{\ell k}}{2}\Bigr) .
\end{multline*}
Note that the second sum is strictly negative for $\gamma \in {[0,\pi/2-\subscr{\varphi}{max}[}$ and can be neglected. Moreover, in the third sum the maximum over all nodes $\ell$ can be taken. Reversing the trigonometric identity from above as $2 \sin(x) \cos(y) = \sin(x-y) + \sin(x+y)$ yields then the simple expression
\begin{align*}
	D^{+} V (\theta(t))
	\leq&\,
	\max_{\{i,j\}} \Bigl| \frac{\omega_{i}}{D_{i}}  - \frac{\omega_{j}}{D_{j}}  \Bigr| - 
	\frac{\subscr{P}{min}}{\subscr{D}{max}} \sum\nolimits_{k=1}^{n}  
	\sin(\gamma ) +
	\max_{i} \sum\nolimits_{k=1}^{n} \frac{P_{i k}}{D_{i}}
	\sin(\varphi_{i k})
	\,.
\end{align*}
It follows that the length of the arc formed by the angles is
non-increasing in $\Delta(\gamma)$ if
\begin{equation}
	\subscr{P}{min}\sin(\gamma)
	\geq
	\subscr{P}{critical} \cos(\subscr{\varphi}{max})
	\label{eq: critical bound - Appendix - 3}
	\,,
\end{equation}
where $\subscr{P}{critical}$ is as stated in equation \eqref{eq: key-assumption - Kuramoto - Appendix - 3}. The left-hand side of \eqref{eq: critical bound - Appendix - 3} is a strictly increasing function of $\gamma \in {[0,\pi/2-\subscr{\varphi}{max}[}$. Therefore, there exists some $\gamma^* \in {[0,\pi/2-\subscr{\varphi}{max}[}$ satisfying equation \eqref{eq: critical bound - Appendix - 3} if and only if equation \eqref{eq: critical bound - Appendix - 3} at $\gamma=\pi/2-\subscr{\varphi}{max}$ is true with the strict inequality sign, which corresponds to equation \eqref{eq: key-assumption - Kuramoto - Appendix - 3}. Additionally, if these two equivalent  statements are true, then there exists a unique $\subscr{\gamma}{min}\in {[0,\pi/2-\subscr{\varphi}{max}[}$ that satisfies equation \eqref{eq: critical bound - Appendix - 3} with the equality sign, namely $\subscr{\gamma}{min} = \arcsin(\cos(\subscr{\varphi}{max}){\subscr{P}{critical}}/{\subscr{P}{min}})$\dots
\hspace*{\fill}~\QED\par\endtrivlist\unskip